\documentclass[12pt]{article}
\usepackage{amsfonts, amsmath}
\usepackage{amssymb}
\usepackage{amsmath,amscd}
\usepackage{lipsum}
\usepackage{enumerate}
\usepackage[pdftex]{graphicx}

\usepackage[margin=25pt,font=small,labelfont=bf]{caption}
\newcommand{\bna}{\begin{eqnarray}}
\newcommand{\ena}{\end{eqnarray}}
\newcommand{\ba}{\begin{eqnarray*}}
\newcommand{\ea}{\end{eqnarray*}}
\newcommand{\bs}[1]{}
\newtheorem{theorem}{Theorem}[section]

\newtheorem{corollary}[theorem]{Corollary}
\newtheorem{lemma}[theorem]{Lemma}
\newtheorem{proposition}[theorem]{Proposition}
\newtheorem{remark}[theorem]{Remark}
\newtheorem{definition}[theorem]{Definition}

\numberwithin{equation}{section}
\numberwithin{figure}{section}

\def\p{{\bf p}}
\def\q{{\bf q}}

\def\p{{\bf p}}

\def\pn{{\bf p =(p_1, \dots, p_n) }}
\def\q{{\bf q}}

\def\vv{{\bf v}}

\def\x{{\bf x}}

\def\0{{\bf 0}}
\def\int{{\text{int}}}

\def\R{\mathbb R}

\def\V{\mathcal V}
\def\M{\mathcal M}
\def\RP{\mathbb R \mathbb P }

\begin{document}
\title{Universal Rigidity of Complete Bipartite Graphs}
\renewcommand\footnotemark{}
\author{Robert Connelly and  
Steven J. Gortler 
\thanks{This work was partially supported by NSF grants  DMS-1564493 and DMS-1564473.}}

\maketitle 

\begin{abstract}  
We describe a very simple condition that is necessary
for the universal rigidity of a complete bipartite framework 
$(K(n,m),\p,\q)$.
This condition is also
sufficient for universal rigidity under a variety of weak assumptions,
such as general position.  Even without any of these assumptions, in
complete generality, we extend these ideas to obtain an efficient
algorithm, based on a sequence of linear programs, that determines
whether an input framework of a complete bipartite graph
is universally rigid or not.


\end{abstract}
\section{Introduction and definitions} \label{section:introduction}

\subsection{Main Results}

A bar and joint 
framework, denoted as $(G,\p)$, is a graph $G$ together with a
configuration $\p=(\p_1, \dots,
\p_N)$ of points in $\R^d$.
A bar and joint framework is \emph{universally rigid} if it is rigid
in any Euclidean space that contains it.  This is equivalent to the
property that 
the framework must be 
congruent to any other configuration of the vertices of the underlying
graph, in any dimension, 
whenever the corresponding edge lengths are the same.  

In this paper, we provide a complete characterization of which
realizations of a complete bipartite graph, $G=K(n,m)$, are
universally rigid and which realizations are not.  As a necessary condition,
we show (Theorem  \ref{thm:main}) that,
except for $K(1,0)$ (a single vertex) and 
$K(1,1)$, if the partitions can be strictly separated by a quadric
surface, then the framework is not universally rigid.  Conversely,
as a sufficient condition, we show
(Corollary \ref{cor:general-position}) that, if the vertices of the
configuration are in general (affine) position in $\R^d$ and there is no
quadric surface strictly separating the partitions, then the framework
is universally rigid.  
Alternatively
 (Corollary
\ref{cor:quadric-general-position}) 
if there are at least
$(d+1)(d+2)/2+1$ vertices and no $(d+1)(d+2)/2$ of them
lie in a quadric surface and if the partitions cannot be strictly
separated by a quadric surface, then the framework is universally rigid.

Even without any of these general position assumptions, in complete
generality, we extend these ideas to obtain an efficient algorithm,
based on a sequence of linear programs, that determines whether an
input framework of a complete bipartite graph is universally rigid or
not.

Surprisingly, 
our results are closely related  to some older statements about
extremal correlation matrices due to 
Tsirelson, which arose in his study of quantum Bell-type 
inequalities. In particular, our  Theorem~\ref{thm:main} 
is related to Tsirelson's Theorem 2.21(c)
in~\cite{ts},
while our Corollary
\ref{cor:quadric-general-position} is related to
Tsirelson's
Theorem 2.22 in~\cite{ts}. Proofs for these 
statements were not
provided in~\cite{ts}.
This connection was pointed out in 
recent work~\cite{GLL} 
which also showed that one can indeed prove Tsirelson's statements
using our results and techniques (which were posted in an
earlier draft of the current paper).

A closely related
concept to universal rigidity is \emph{global rigidity} in $\R^d$, which is 
similar except that the
other configurations $\q$,
where corresponding edge lengths are the same,
are restricted to be in $\R^d$.  Clearly, if
$(G,\p)$ is universally rigid, then it is automatically globally
rigid in $\R^d$.  But in 
most results, for a framework to be globally rigid,  it is assumed
that the configuration $\p$ is \emph{generic}, which means that there
is no non-zero integral polynomial relation among the coordinates of
$\p$, and it may be very hard to verify that some specific framework
is acting  generically.
So, when possible, the stronger condition of
universal rigidity can be a useful condition that is sufficient 
to show that 
particular configuration is  globally rigid in $\R^d$.

\subsection{Definitions}

 The basic tool we use in this paper is a stress $\omega =( \dots,
 \omega_{ij}, \dots )$, which is an assignment of a real scalar
 $\omega_{ij}=\omega_{ji}$ to each edge, $\{i,j\} \in E(G)$.  We assume
 $\omega_{ij}=0$, when $\{i,j\} \not\in E(G)$.  We say that a stress
 $\omega$ 
is an \emph{equilibrium} stress for $(G,\p)$ 
if the vector equation
\begin{equation}\label{eqn:equilibrium}
 \sum_i \omega_{ij}(\p_i-\p_j) = 0
 \end{equation}
holds for all vertices $j$ of $G$.  We associate an $N$-by-$N$
\emph{stress matrix} $\Omega$ to a stress $\omega$, for $N$, the total
number of vertices, by saying that $i,j$ entry of $\Omega$ is
$-\omega_{ij}$, for $i \ne j$, and the diagonal entries of $\Omega$
are such that the row and column sums of $\Omega$ are zero.  If the
dimension of the affine span of the vertices $\p$ is $d$, then the
rank of an equilibrium stress matrix 
$\Omega$ is at most $N-d-1$, but it could be less.

We say that $\vv= \{\vv_1, \dots, \vv_m\}$, a finite
collection of non-zero vectors in $\R^d$, lie on a \emph{conic at
  infinity} of $\R^d$ if, when regarded as points in real projective $(d-1)$
space $\RP^{d-1}$, they lie on a conic.  This means that there is a
non-zero $d$-by-$d$ symmetric matrix $Q$ such that for all $i = 1,
\dots, m$, $\vv_i^tQ\vv_i=0$, where $()^t$ is the transpose operation.

\subsection{Basic Results}

\begin{definition} 
A framework $(G,\p)$ in $\R^d$ is said 
to be \emph{universally rigid} if any
other corresponding framework with the same edge lengths in $\R^D$,
for any $D$, is congruent to $(G,\p)$.
\end{definition}

The following fundamental theorem \cite{connelly-energy} is a basic
tool used to establish  universal rigidity.

\begin{theorem} \label{thm:fundamental}  Let $(G,\p)$ be a framework 
whose affine span of $\p$ is all of $\R^d$, with an equilibrium stress
$\omega$ and stress matrix $\Omega$.  Suppose further
\begin{enumerate}[(i)]
\item \label{condition:positive} $\Omega$ is positive semi-definite (PSD).
\item \label{condition:rank} The rank of $\Omega$ is $N - d - 1$.
\item \label{condition:conic} The edge directions of $(G,\p)$ do not lie on a conic at infinity of $\R^d$.
\end{enumerate}
Then $(G,\p)$ is universally rigid.
\end{theorem}

There are several examples of the universally rigid frameworks in
\cite{Connelly-Gortler-iterated}, where Conditions
(\ref{condition:positive}) and (\ref{condition:rank}) of Theorem
\ref{thm:fundamental} do not hold, and yet they are still universally
rigid.

\begin{definition} When conditions (i), (ii), (iii) 
hold for a framework $(G,\p)$ with  affine span  $\R^d$,  
we say it is \emph{super stable}.  
\end{definition}

\begin{remark}
If a framework $(G,\p)$ in $\R^d$ 
happens to have an  affine span of some smaller dimension $d'$, then
the framework can be rigidly placed in $\R^{d'}$ and 
Theorem~\ref{thm:fundamental}
can be applied if appropriate. In this case, we also say that 
$(G,\p)$ is super stable.
\end{remark}

When
the sign of the stress $\omega_{i,j}$ is positive (respectively
negative), then the constraint on the lengths of the edges of the
possible alternative configurations $\q$ can be weakened to be not
longer (respectively not shorter), and the conclusion 
of 
Theorem~\ref{thm:fundamental}
still
holds. Those edges with a positive stress are called \emph{cables} and
those with a negative stress are called 
\emph{struts}.  When possible, in the
following, we will designate cables with dashed line segments and
struts with heavy solid line segments.  The default is that the edge
lengths are constrained to stay the same length.

\begin{definition} 
A framework $(G,\p)$ in $\R^d$
with a $d'$-dimensional span is said 
to be \emph{dimensionally rigid} if any
other corresponding framework with the same edge lengths in $\R^D$,
for any $D$, has affine span at most $d'$.
\end{definition}

For completeness, we note the following from~\cite{alfakih2014local}:
\begin{theorem}\label{thm:Alfakih-dim-rigid} 
If $(G,\p)$ is a dimensionally rigid framework in $\R^d$
with $N$ vertices whose
affine span is $d$-dimensional
and Condition (\ref{condition:conic}) of Theorem \ref{thm:fundamental} holds, then $(G,\p)$ is universally rigid.
\end{theorem}

The following result in \cite{Alfakih-bar-frameworks} 
will be important in this paper.
(See also
\cite{Connelly-Gortler-iterated} for another point of view for the
proof.)

\begin{theorem}\label{thm:Alfakih-stress} 
If $(G,\p)$ is a dimensionally rigid or 
universally rigid framework in $\R^d$
with $N$ vertices whose
affine span is $d'$-dimensional, $d' \le N-2$, then it has a non-zero 
equilibrium
stress with a positive semi-definite (PSD) stress matrix $\Omega$,
(with rank $\geq 1$).
\end{theorem} 

So, in particular, if  $(G, \p)$ has no non-zero PSD 
equilibrium stress matrix, and the dimension of the affine span 
of $\p$ is $\le N-2$, then it is not universally rigid.

\section{Bipartite frameworks and quadrics} \label{section:quadrics}

Let $K(n,m)$ be the complete bipartite graph on $n$ and $m$ vertices, and let 
$\pn$ and $\q=(\q_1, \dots, \q_m)$ be  two
configurations of points in $\R^d$. 
Then we denote by 
$(K(n,m),\p,\q)$ the associated complete bipartite framework.

Recall that a quadric surface in $\R^d$, is the solution to a non-zero
quadratic function in the coordinates of $\R^d$.  For the line $\R$,
a quadric surface is two points.  For the plane $\R^2$, a quadric
surface
is a conic, which includes the possibility of two straight lines as
well an ellipses and a hyperbola.  (We do not need to consider quadrics
that consist of just one hyperplane in $\R^d$.)  By adjoining the
projective space $\RP^{d-1}$, we can complete $\R^d$ to real
projective space $\RP^{d}$, and a quadric will separate $\RP^{d}$ into
two components. For any vector $\x \in \R^d$, define $\hat{\x} \in
\R^{d+1}$ by adding a $1$ as the last coordinate.  
A quadric can be written in the
form $\{\x \in \R^d \mid \hat{\x}^tA\hat{\x}=0\}$, where $A$ is a
$(d+1)$-by-$(d+1)$ symmetric matrix, $\hat{\x}$ is a column vector,
and $\hat{\x}^t$ is its transpose.  So the two components determined
by the matrix $A$ are given by $\hat{\x}^tA\hat{\x}< 0$ and $0 <
\hat{\x}^tA\hat{\x}$.  
\begin{definition} If $\pn$ and $\q=(\q_1, \dots, \q_m)$ are two
configurations of points in $\R^d$, we say that they are
\emph{strictly separated by a quadric}, given by a matrix $A$, if for
each $i=1, \dots, n$ and $j=1, \dots, m$,
\begin{equation}\label{eqn:quadric}
 \hat{\q}_j^tA\hat{\q}_j< 0 <\hat{\p}_i^tA\hat{\p}_i.
\end{equation} 
\end{definition}

A stress matrix for a complete bipartite framework $(K(n,m),\p,\q)$,
has the following form, where $\lambda_1, \dots, \lambda_n$ and $\mu_1, \dots, \mu_m$ are the diagonal entries, whereas all the non-diagonal entries in the upper left and lower right blocks are zero. 
\begin{eqnarray}\label{eqn:stress-matrix}
\Omega =
\begin{pmatrix}
\lambda_1 & 0 & 0 & -\omega_{11} & \cdots & -\omega_{1m}\\
 0 &  \ddots & 0 & \vdots & \ddots & \vdots\\
 0 &  0 & \lambda_n & -\omega_{n1} & \cdots & -\omega_{nm}\\
- \omega_{11} & \cdots & -\omega_{n1} & \mu_1 & 0 & 0 \\
 \vdots & \ddots & \vdots & 0 &  \ddots & 0 \\
  -\omega_{1m} & \dots & -\omega_{nm} & 0 &  0 & \mu_m
\end{pmatrix}
\end{eqnarray}
So the diagonal entries are such that $\sum_{j=1}^m \omega_{ij}= \lambda_i$, and $\sum_{i=1}^n  \omega_{ij}= \mu_j$, from the definition of $\Omega$.

Our first main result is the following necessary condition for the 
universal rigidity of a complete bipartite framework.

\begin{theorem}\label{thm:main} 
If $(K(n,m),(\p,\q))$ is a complete bipartite framework in $\R^d$,
with an affine span of dimension $\leq m+n-2$,
such that the partition vertices $(\p,\q)$ are
strictly separated by a quadric, then it is not universally rigid.
\end{theorem} 

\begin{proof}  Let $A$ be the  $(d+1)$-by-$(d+1)$ symmetric matrix for the separating quadric as above, and let $\omega$ be any equilibrium stress for $(K(n,m),(\p,\q))$ with stress matrix $\Omega$.   For any vertex $\q_j$ in one partition, the equilibrium condition 
of Equation (\ref{eqn:equilibrium}) can be written, for each $j=1, \dots, m$ as 
\[
 \sum_{i=1}^n \omega_{ij}(\hat{\p}_i-\hat{\q}_j) = 0,
\]
or equivalently
\[
 \sum_{i=1}^n \omega_{ij}\hat{\p}_i= \left( \sum_{i=1}^n \omega_{ij} \right) \hat{\q}_j = \mu_j \hat{\q}_j.
\]
Then taking the transpose of this equation, and multiplying on the right by $A\q_j$, we get  
\[
 \sum_{i=1}^n \omega_{ij}\hat{\p}^t_iA\hat{\q}_j= \mu_j \hat{\q}^t_jA\hat{\q}_j.
\]
Similarly, for  $\p_i$ in the other partition,
\[
 \sum_{j=1}^m \omega_{ij}\hat{\q}^t_jA\hat{\p}_i
= \lambda_i \hat{\p}^t_iA\hat{\p}_i.
\]
Since the matrix $A$ is symmetric, 
\begin{equation}\label{eqn:equal}
 \sum_{j=1}^m \mu_j \hat{\q}^t_jA\hat{\q}_j=
\sum_{ij} \omega_{ij}\hat{\p}^t_iA\hat{\q}_j=
\sum_{ij} \omega_{ij}\hat{\q}^t_jA\hat{\p}_i= 
\sum_{i=1}^n\lambda_i \hat{\p}^t_iA\hat{\p}_i.
\end{equation}
By Theorem \ref{thm:Alfakih-stress}, if $(K(n,m),(\p,\q))$ were
universally rigid, then there would be an equilibrium stress with a
stress matrix $\Omega$ that would be PSD and non-zero.  Then $\mu_j
\ge 0$ for all $j= 1, \dots, m$, $\lambda_i \ge 0$ for all $i= 1,
\dots, n$, and we would have at least one positive diagonal term. But
then Equation~(\ref{eqn:equal}) would
contradict the assumed quadric separation condition 
of Equation~(\ref{eqn:quadric}).~$\Box$
\end{proof}\\

This result is a generalization of, and inspired by,
 the main result in \cite{Jordan-universal-line}, which is the result here for the line $d=1$.  
We will see that the quadric separation condition is 
also the critical sufficient 
condition for complete bipartite graphs to be universally rigid, 
including in higher dimensions, but we need to use a technique that allows us to find PSD matrices with a given kernel and appropriate rank,
which we describe in later sections. 
 
\section{The Veronese map} \label{section:Veronese}

Vectors in $\R^d$
will be regarded as column vectors, and in general, for any vector or matrix $X$, we will denote by $\hat{X}$ the same object with a row of $1$'s added on the bottom.   We denote $X^t$ as the transpose of  a matrix or vector  $X$.

If two configurations $\pn$ and $\q=(\q_1, \dots, \q_m)$ in $\R^d$
cannot be separated by a quadric, 
i.e. when 
the condition of Equation 
(\ref{eqn:quadric}) cannot be made to hold for any $A$,
we show here how find a certificate
of this non-separability, 
that can help us to establish universal rigidity.  

\begin{definition}
We define 
$\M_d$, to be the $(d+1)(d+2)/2$-dimensional space of $(d+1)$-by-$(d+1)$
symmetric matrices, which we call the \emph{matrix space}.
\end{definition}

\begin{definition}
We define the map $\V: \R^d \rightarrow \M_d$ by $\V(\vv)
=\hat{\vv}\hat{\vv}^t$, which is a $(d+1)$-by-$(d+1)$ symmetric
matrix, with a lower right-hand coordinate of $1$.
\end{definition}
So $\V( \R^d)$ is a $d$-dimensional set  embedded in a
$(d+1)(d+2)/2-1$-dimensional affine subspace of $\M_d$. The function $\V$
is called the \emph{Veronese map}.  See page 244 of \cite{Matousek},
for very similar properties that are used here.

\begin{proposition}\label{prop:separation} In $\R^d$ the vertices of the configurations $\p$ and $\q$ can be strictly separated by a quadric $A$ as in Section \ref{section:quadrics}, if and only if the matrix configurations
$\V(\p)$ and  $\V(\q)$ can be strictly separated by the the hyperplane given by $A$ in $\M_d$.
\end{proposition}
\begin{proof}
The configurations $\p$ and $\q$ are separated by the quadric given by the matrix $A$ when 
\[  \langle A, \V(\q_j) \rangle  =
\text{tr}(A\hat{\q}_j\hat{\q}_j^t) = 
\hat{\q}_j^tA\hat{\q}_j  < 0 < \hat{\p}_i^tA\hat{\p}_i = 
\text{tr}(A\hat{\p}_i\hat{\p}_i^t) = 
\langle A, \V(\p_i) \rangle,
\]
where the inner product $ \langle * , * \rangle$ on symmetric matrices is given by the trace operator $\text{tr}$ as above.   $\Box$  \end{proof}

When the configurations $\p$ and $\q$ cannot be separated by a quadric
in $\R^d$, then from Proposition
\ref{prop:separation}
the convex hull of $\V(\p)$ must
intersect the convex hull of $\V(\q)$ in $\M_d$.
This means that there are 
non-negative coefficients, not all $0$,  denoted as
$\lambda_1, \dots, \lambda_n$ and $\mu_1, \dots, \mu_m$, such that
\begin{equation}\label{eqn:Radon1}
\sum_{i=1}^n \lambda_i \hat{\p}_i \hat{\p}_i^t= 
\sum_{j=1}^m \mu_j \hat{\q}_j \hat{\q}_j^t,
\end{equation}

\begin{definition}
The matrix $\hat{P}$ whose columns are $\hat{\p}_1, \dots, \hat{\p}_n$ is called the \emph{configuration matrix} of $\p$.
\end{definition}

In terms of matrices Equation (\ref{eqn:Radon1}) is the same as saying:
\begin{equation}\label{eqn:Radon}
\hat{P}\Lambda \hat{P}^t=\hat{Q} M \hat{Q}^t,
\end{equation}
 where $\Lambda$ is the $n$-by-$n$ diagonal matrix whose entries are $\lambda_1, \dots, \lambda_n$, and $M$ is the $m$-by-$m$ diagonal matrix whose entries are $\mu_1, \dots, \mu_m$.  This is the starting point for constructing a 
PSD stress matrix in Section \ref{section:SVD}.

\section{The Singular Value Decomposition} \label{section:SVD}
We first show that 
when each of the partition's affine span is
the full $\R^d$, we will not need to worry about  
condition (iii) of Theorem~\ref{thm:fundamental}.

\begin{lemma}\label{lemma:proper-affine}
Suppose that the configurations $\p$ and $\q$ in $\R^d$
each have affine span
equal to all of $\R^d$, and the bipartite framework $(K(n,m),(\p,\q))$
has a stress with stress matrix, $\Omega$, satisfying
(\ref{condition:positive}) and (\ref{condition:rank}) of Theorem
\ref{thm:fundamental}.  Then $(K(n,m),(\p,\q))$ is super stable and
universally rigid.  Likewise, if, instead of (\ref{condition:positive}) and
(\ref{condition:rank}) of Theorem \ref{thm:fundamental},
$(K(n,m),(\p,\q))$ is just dimensionally rigid, then it is still
universally rigid.
\end{lemma}

\begin{proof}
We have only to check (\ref{condition:conic}) of Theorem
\ref{thm:fundamental}, that the edge directions do not lie on a conic
at infinity of $\R^d$.  Suppose there is a non-zero symmetric
$d$-by-$d$ matrix $Q$ such that $(\p_1-\q_j)Q(\p_1-\q_j)^t=0$ and
$(\p_2-\q_j)Q(\p_2-\q_j)^t=0$, for all $j=1, \dots, m$.  Expanding
these terms and subtracting we get
\begin{eqnarray*}
\p_1Q\p^t_1 - 2 \p_1Q\q^t_j &-& \p_2Q\p^t_2 + 2 \p_2Q\q^t_j =0,\,\,\,\, 
\text{which gives us}\\
\p_1Q\p^t_1 -  \p_2Q\p^t_2 &=&2(\p_1-\p_2)^tQ\q^t_j ,
\end{eqnarray*}
which is a non-trivial affine linear 
constraint on the vertices of $\q$, unless for all $\p_i$ and $\p_k$,  
\[
(\p_i-\p_k)^tQ=0
\]
The first case 
implies that the vertices of $\q$ lie in a proper affine subspace, while
the latter
implies that the vertices of $\p$ lie in a proper affine subspace. $\Box$
\end{proof}

Alfakih and Ye in \cite{Alfakih-Ye-general-position} show that if a
configuration of a framework is in general position and satisfies
(\ref{condition:positive}) and (\ref{condition:rank}) of Theorem
\ref{thm:fundamental}, then it is universally rigid.  Lemma
\ref{lemma:proper-affine}
is more precise and general for complete
bipartite graphs.

 For any diagonal matrix $X$, with non-negative entries, denote $X^{1/2}$ as another diagonal matrix whose entries are the square roots of the entries of $X$.  

\begin{definition}
For any $a$-by-$b$ matrix $X$, $a \le b$, a \emph{singular value
  decomposition} (SVD) is a factoring $X = U S V^t$, where $U$ is an
$a$-by-$a$ orthogonal matrix, $V$ is a $b$-by-$b$ orthogonal matrix,
and $S$ is the matrix of 
$S=[D, 0]$, where $D$ is an
$a$-by-$a$ diagonal matrix of non-negative
$\emph{singular values}$.
  Such a decomposition always exists (see
e.g.~\cite{Horn-Johnson}).
\end{definition}

Our next step is to show that when Equation
(\ref{eqn:Radon}) holds,  
$\hat{P}\Lambda^{1/2}$ and $\hat{Q}M^{1/2}$ must share their singular values
and their left singular structure.

\begin{lemma}\label{lemma:factor} Suppose Equation
(\ref{eqn:Radon}) holds where $\Lambda$ and $M$ are non-negative,
diagonal matrices as above.  Then the SVD factors can be taken such that 
$\hat{P}\Lambda^{1/2} = U S_n V^t_n$, and $\hat{Q}M^{1/2} = U S_m V^t_m$,
with a common matrix $U$ and  where $S_n S_n^t=S_m S_m^t$.
\end{lemma}

\begin{proof}
By definition,
the squared singular values and the left singular vectors of
$\hat{P}\Lambda^{1/2}$
are the eigenvalues and eigenvectors of
$\hat{P} \Lambda \hat{P}^t$.
Likewise, the squared singular values and the left singular vectors of
$\hat{Q}M^{1/2}$
are the eigenvalues and eigenvectors of
$\hat{Q} M \hat{Q}^t$.

Since, by assumption,
$\hat{P} \Lambda \hat{P}^t =  \hat{Q} M \hat{Q}^t$,
these singular values and left singular vectors agree.
Thus we can pick a single shared $(d+1)$-by-$(d+1)$ matrix $U$, along with with appropriately
sized diagonal matrices $S_n$ and $S_m$, and appropriate
orthogonal matrices
$V_n$ and $V_m$,
such that we obtain the singular value decompositions:
\[
\hat{P}\Lambda^{1/2} = U S_n V_n^t \,\,\,\,\text{and}\,\,\,\, \hat{Q}M^{1/2} = U S_m V_m^t
\]
where $S_n S_n^t = S_m S_m^t$.  In particular
\[
S_n = D[I^{d+1},0^{n-d-1}]  \,\,\,\,\text{and}\,\,\,\, S_m = D[I^{d+1},0^{m-d-1}]
\]
for a single shared diagonal matrix $D$. $\Box$
\end{proof}

Our next result is our main sufficient condition for the universal
rigidity of a complete bipartite framework. The central idea is to use
the conditions of Equation (\ref{eqn:Radon}) to directly construct
$\Omega$, a PSD equilibrium stress matrix for $(K(n,m),\p,\q)$ that has
rank $n+m-d-1$.  To do this we will use the SVD provided by 
Lemma~\ref{lemma:factor} in
order to \emph{transform} the matrix $[\hat{P},\hat{Q}]$ into a very
specific and simple canonical form. It will be easy to see that this
canonical form is annihilated by a certain simple PSD matrix $\Psi$
described below. We can then reverse this transformation, thus
constructing a $\Omega$ with the same signature as $\Psi$.

\begin{theorem}\label{thm:SVD} 
Let $\p$ and $\q$ be configurations 
(in any dimension), 
such that
Equation~(\ref{eqn:Radon1}) holds with strictly positive coefficients.  Then
the framework $(K(n,m),\p,\q)$ is super stable, and thus universally
rigid.  Additionally, the affine span of $\p$ is the same as the
affine span of $\q$.
\end{theorem}

\begin{proof}
Let $d$ be the dimension of the
combined span of $(\p,\q)$.
Without loss of generality,
we can 
rigidly place $(\p,\q)$ in $\R^{d}$ and continue.

By Lemma \ref{lemma:factor} we have 
the following $(d+1)$-by-$(n+m)$ matrix equality:
\[
[\hat{P}\Lambda^{1/2}, \hat{Q}M^{1/2}]=[ U S_n V^t_n, U S_m V^t_m],
\]
 where  $U, V_n, V_m$ are orthogonal matrices, of the appropriate size, \\$S_n=D[I^{d+1},0^{n-d-1}]$, and $S_m=D[I^{d+1},0^{m-d-1}]$, where $D$ is a $(d+1)$-by-$(d+1)$ diagonal matrix, 
$I^{d+1}$ is the  $(d+1)$-by-$(d+1)$ identity matrix, $0^{n-d-1}$ is a $(d+1)$-by-$(n-d-1)$ zero matrix and $0^{m-d-1}$ is a $(d+1)$-by-$(m-d-1)$ zero matrix.  Then
\begin{eqnarray}\label{eqn:SVD}
[\hat{P}, \hat{Q}]&=&[\hat{P}\Lambda^{1/2}, \hat{Q}M^{1/2}]
\begin{bmatrix}
\Lambda^{-1/2} & 0\\
0 & M^{-1/2}
\end{bmatrix}\\
&=& UD[I^{d+1},0^{n-d-1}, I^{d+1},0^{m-d-1}]
\begin{bmatrix}
V^t_n & 0\\
0 & V^t_m
\end{bmatrix}
\begin{bmatrix}
\Lambda^{-1/2} & 0\\
0 & M^{-1/2}
\end{bmatrix}.\nonumber
\end{eqnarray}
(The matrices
 $\Lambda^{-1/2}$ and $M^{-1/2}$ are well defined due to our assumption
of strictly positive coefficients.)

Define the following symmetric $(n+m)$-by-$(n+m)$ matrix
\begin{equation}\label{eqn:SVD2}
\Psi=
\begin{bmatrix}
~I^{d+1} &  & -I^{d+1} &  \\
 & I^{n-d-1} & & \\
 -I^{d+1} &   & ~I^{d+1} &  \\
  &   &  & I^{m-d-1}
\end{bmatrix},
\end{equation}
where the blank entries are zero matrices of the appropriate dimensions.  It is easy to check that 
\[
[I^{d+1},0^{n-d-1}, I^{d+1},0^{m-d-1}]\Psi=0,
\]
and that $\Psi$ is PSD of rank $n+m-d-1$.  Then we define a stress matrix
\begin{equation} \label{eqn:stress-def}
\Omega = 
\begin{bmatrix}
\Lambda^{1/2} & 0\\
0 & M^{1/2}
\end{bmatrix}
\begin{bmatrix}
V_n & 0\\
0 & V_m
\end{bmatrix}
\Psi
\begin{bmatrix}
V^t_n & 0\\
0 & V^t_m
\end{bmatrix}
\begin{bmatrix}
\Lambda^{1/2} & 0\\
0 & M^{1/2}
\end{bmatrix}.
\end{equation}
Clearly $\Omega$ has zero entries for all of the non edges of the
complete bipartite graph.  Thus by unraveling Equations (\ref{eqn:SVD}),
(\ref{eqn:SVD2}) and (\ref{eqn:stress-def}),
 and using the assumption
that the diagonal entries of $\Lambda$ and $M$ are all positive, 
we see that 
$\Omega$ is PSD of rank
$n+m-d-1$, and $[\hat{P}, \hat{Q}]\Omega=0$.
This is sufficient to obtain
conditions (\ref{condition:positive}) and (\ref{condition:rank}) of
Theorem \ref{thm:fundamental}.

The equilibrium condition of Equation
(\ref{eqn:equilibrium}) at each vertex, and the non-zero diagonal entries in the stress matrix, imply that each $\p_i$ is in the affine span of $\q$, and similarly each $\q_j$ is in the affine span of $\p$.  
So the affine span of $\p$ is the same as the affine span of $\q$, which
by our assumptions must then be all of $\R^d$.  
Lemma \ref{lemma:proper-affine} then implies that condition (\ref{condition:conic}) of Theorem \ref{thm:fundamental} holds. 
$\Box$
\end{proof}

The next two corollaries describe partial converses to our Theorem \ref{thm:main}, each requiring
some kind of general position for the configuration. Without any such assumptions,
the converse of Theorem \ref{thm:main} does not hold. In Section 
\ref{section:algorithm}, we use our Theorem \ref{thm:SVD} as the basis of
a complete algorithm
for determining the universal and dimensional rigidity of any 
complete bipartite framework.

\begin{definition} A configuration $\p$ in $\R^d$
is in \emph{general position} if every $k+1$ of the points of $\p$ span a $k$-dimensional affine subspace for $k = 1, \dots, d$.  
\end{definition}

\begin{corollary}\label{cor:general-position}  
Let $\p$ and $\q$ be configurations in $\R^d$.
Suppose there exists 
subsets of the corresponding configurations $\p'
\subset \p$ and $\q' \subset \q$, such that 
that the points of $(\p',\q')$ are in general position in
$\R^d$, and such that there 
is no quadric strictly
separating $\p'$ and $\q'$.  
Then $(K(n,m),\p,\q)$ is universally
rigid.  Additionally the affine span of $\p$ and the affine span of 
$\q$ must be all of $\R^d$.
\end{corollary}

\begin{proof}  
Since there is no quadric strictly separating $\p'$ and $\q'$, the
convex hulls $\V(\p')$ and $\V(\q')$ must intersect in matrix space,
$\M_d$, and Equation (\ref{eqn:Radon1}) holds with strictly positive
coefficients $\lambda_i$, and $\mu_i$ for some subsets $\p'' \subset
\p'$ and $\q'' \subset \q'$.  By Theorem \ref{thm:SVD}, that
subframework is super stable.  

Also from Theorem~\ref{thm:SVD}
each vertex of $\p''$ must be in the affine span 
of the $\q''$, and so due to 
general position
assumption, the affine span of $\q''$ must then be $d$-dimensional.
Since each of the vertices of $\p$
has at least $d+1$ neighbors in $\q''$, each 
$\p$ has a fixed distance to all the
vertices of $\q''$.  The same argument applies to $\q$.
This \emph{trilateration} argument shows that all of
$(K(n,m),\p,\q)$ is universally rigid. $\Box$
\end{proof}

Note that it may be the case that, even assuming general position, the
framework is not super stable, because all the stress coefficients may
vanish for some vertex.  See the example of Figure \ref{fig:K33} that
shows this possibility, and other examples of universally rigid
frameworks.

\begin{definition} 
We say that a configuration $\p$ in $\R^d$ is in \emph{quadric general
  position} if every $k+1$ of the points of $\V(\p) \subset \M_d$ span
a $k$-dimensional affine subspace for $k = 1, \dots, (d+1)(d+2)/2-1$.
(The vertices of $\V(\p)$ are automatically mapped into a co-dimension
one subspace of $\M_d$, where the last coordinate is one.) Essentially
this means that if there are at least $(d+1)(d+2)/2$ points, then  no
$(d+1)(d+2)/2$ of them lie on a quadric.
\end{definition}

\begin{corollary}\label{cor:quadric-general-position}  
Let $\p$ and $\q$ be configurations in $\R^d$.
Suppose there exists 
subsets of the corresponding configurations $\p'
\subset \p$ and $\q' \subset \q$, such that 
that the points of $(\p',\q')$ are in 
quadric general position in
$\R^d$, and such that there 
is no quadric strictly
separating $\p'$ and $\q'$.  
Then $(K(n,m),\p,\q)$ is super stable,
and $n+m \ge (d+1)(d+2)/2+1$.
Additionally, 
the affine span of $\p$ is the same as the affine span of $\q$,
which must be all of $\R^d$.
\end{corollary}

\begin{proof}  
In matrix space $\M_d$, since the points $\V(\p')$ and $\V(\q')$
cannot be separated by a hyperplane, Equation (\ref{eqn:Radon1}) holds with
non-negative coefficients, not all $0$.  
But since $(\V(\p'),\V(\q')$
is  in general position, then at least $(d+1)(d+2)/2+1$ of the
coefficients are positive, corresponding to subsets $\p'' \subset \p$
and $\q'' \subset \q$.  

This gives us  $n+m \ge  (d+1)(d+2)/2+1$.
Additionally, this lower bound together with
our quadric general position
assumption, forces the combined span of $(\p'',\q'')$ to be all of $\R^d$.

By Theorem \ref{thm:SVD}, the bipartite
graph restricted to $(\p'', \q'')$ is super stable.
Additionally the span of
$\p''$ and the span of $\q''$ must be all of $\R^d$.

Additionally, due to Equation (\ref{eqn:Radon1}) and the quadric
general position assumption, 
the affine span of this $(\V(\p''),\V(\q''))$ in
matrix space must be the full $(d+1)(d+2)/2-1$ dimensions.
Thus, for any additional
point, $\p_i \in \p$, 
there is an affine relation non-zero on
$\V(\p_i)$ and involving the $\V(\p'')$ and  $\V(\q'')$.  
When that relation is added
to both sides of Equation (\ref{eqn:Radon1}), choosing the coefficient
of $\V(\p_i)$ to be positive, and the whole relation small enough, we
enlarge the number of indices, where $\lambda_i > 0$, and $\mu_j > 0$,
until all the coefficients are positive, applying this argument to any
$\q_j \in \q$ as well.  Then again Theorem \ref{thm:SVD} implies that all
of $(K(n,m),\p,\q)$ is super stable.  
$\Box$
\end{proof}

\begin{remark}
The smallest example of Corollary \ref{cor:quadric-general-position}
in the line is $K(2,2)$. 
In the plane, the smallest example is $K(4,3)$.
In $\R^3$, the smallest examples are
$K(7,4)$ and $K(6,5)$.  Section \ref{section:examples} shows some
examples of these.  For the first three examples, 
any equilibrium stress matrix with positive diagonals will be
PSD, which implies directly that they are super stable.  However, for
$K(6,5)$ in $\R^3$, there are always equilibrium 
stress matrices with all positive diagonals but with
negative eigenvalues.
At first it is a little surprising that Theorem~\ref{thm:SVD}
guarantees that there will always be some such PSD stress matrix.
\end{remark}

\begin{remark}We note that Theorem~\ref{thm:SVD} is also gives us an alternative
proof for Theorem 6 of \cite{Bolker-Roth}, under the restriction that
the coefficients $\lambda_1, \dots, \lambda_n$ and $\mu_1, \dots,
\mu_m$ of Equation (\ref{eqn:Radon1}) are positive. This is
because the proof of our Theorem~\ref{thm:SVD} provides a construction
of an equilibrium stress matrix for $(K(n,m),\p,\q)$ with these
coefficients on its diagonal.

Indeed, by slightly generalizing this construction, we can produce
\emph{all} of the equilibrium stress matrices for $(K(n,m),\p,\q)$ with
all positive diagonals. To do this, all we need to do is replace
Equation (\ref{eqn:SVD2}) with
\[
\Psi=
\begin{bmatrix}
~I^{d+1} &  & -I^{d+1} &  \\
& I^{n-d-1} & & C \\
-I^{d+1} &   & ~I^{d+1} &  \\
 & C^t  &  & I^{m-d-1}
\end{bmatrix},
\]
where
$C$
is an arbitrary diagonal
$(n-d-1)$-by-$(m-d-1)$ matrix, and
also we need to allow for any $U$, $V_m$ and $V_n$ such that
\[
\hat{P}\Lambda^{1/2} = U S_n V_n^t \,\,\,\,\text{and}\,\,\,\, \hat{Q}M^{1/2} = U S_m V_m^t
\]

Additionally, whenever any of the diagonal entries in $C$ have a magnitude
equal to $1$ the rank of $\Psi$ will drop, and
whenever all of diagonal entries of $C$ have magnitudes less
than or equal to
$1$, then $\Psi$ will be PSD.

It is less clear if we can use the ideas in this paper to 
prove Theorem 6 of~\cite{Bolker-Roth}, when 
the coefficients
$\lambda_1, \dots, \lambda_n$ and $\mu_1, \dots, \mu_m$
include negative values.
\end{remark}

It is easy to see how our necessity result of Theorem~\ref{thm:main}
fits in with the ideas of this section. In particular we have the following
Proposition,
which is essentially Lemma 5 of~\cite{Bolker-Roth}.
\begin{proposition}
\label{prop:brn}
Suppose that $\Omega$ is an equilibrium stress matrix for\\
$(K(n,m),\p,\q)$, where $\Omega$ is of the form
\[
\begin{bmatrix}
\Lambda & B\\
B^t & M
\end{bmatrix},
\]
where $\Lambda$ and $M$ are diagonal matrices of size $n$ and $m$ respectively.  Then Equation (\ref{eqn:Radon}) holds with this  $\Lambda$ and $M$.
\end{proposition}
\begin{proof}
Since $\Omega$ is an equilibrium stress matrix we have
\[
[\hat{P},\hat{Q}]
\begin{bmatrix}
\Lambda & B\\
B^t & M
\end{bmatrix}
=0
\]
and so we have
$\hat{P} \Lambda = -\hat{Q}B^t$
and
$\hat{Q} M = -\hat{P}B$.
This gives us
$\hat{P} \Lambda \hat{P}^t= -\hat{Q}B^t\hat{P}^t$
and
$\hat{Q} M \hat{Q}^t= -\hat{P}B\hat{Q}^t$.
Since these are symmetric matrices, this gives us
$\hat{P} \Lambda \hat{P}^t=
\hat{Q} M \hat{Q}^t$, which is Equation (\ref{eqn:Radon}).
$\Box$
\end{proof}

Thus, when $(K(n,m),\p,\q)$ is universally rigid,
from Theorem~\ref{thm:Alfakih-stress}
it must have
a non-zero equilibrium stress matrix with 
non-negative (and not all zero)  $\Lambda$ and $M$.
Then Propositions~\ref{prop:brn} and~\ref{prop:separation} imply that
$\p$ and $\q$ cannot be strictly separated by a quadric, which gives us
the result of our Theorem~\ref{thm:main}.

\newpage

\section{Examples} \label{section:examples}

Figure \ref{fig:min} shows examples of bipartite frameworks that are super stable in quadric general position with the minimal number of vertices.  Dashed edges have a positive equilibrium stress, and for solid edges the equilibrium stress is negative.  These represent cables and struts, respectively, where cables cannot increase in length, and struts cannot decrease in length.  These examples have symmetry, and for the calculation of the separating quadric or conic, this allows us to only consider symmetric quadrics or conics, since we can average those that separate the two partitions to get one that is symmetric.  Note that the $K(6,5)$ example is such that it is in quadric general position, but since there are several sets of  three vertices that collinear, it is not in general position.
\begin{figure}[here]
    \begin{center}
        \includegraphics[width=0.9\textwidth]{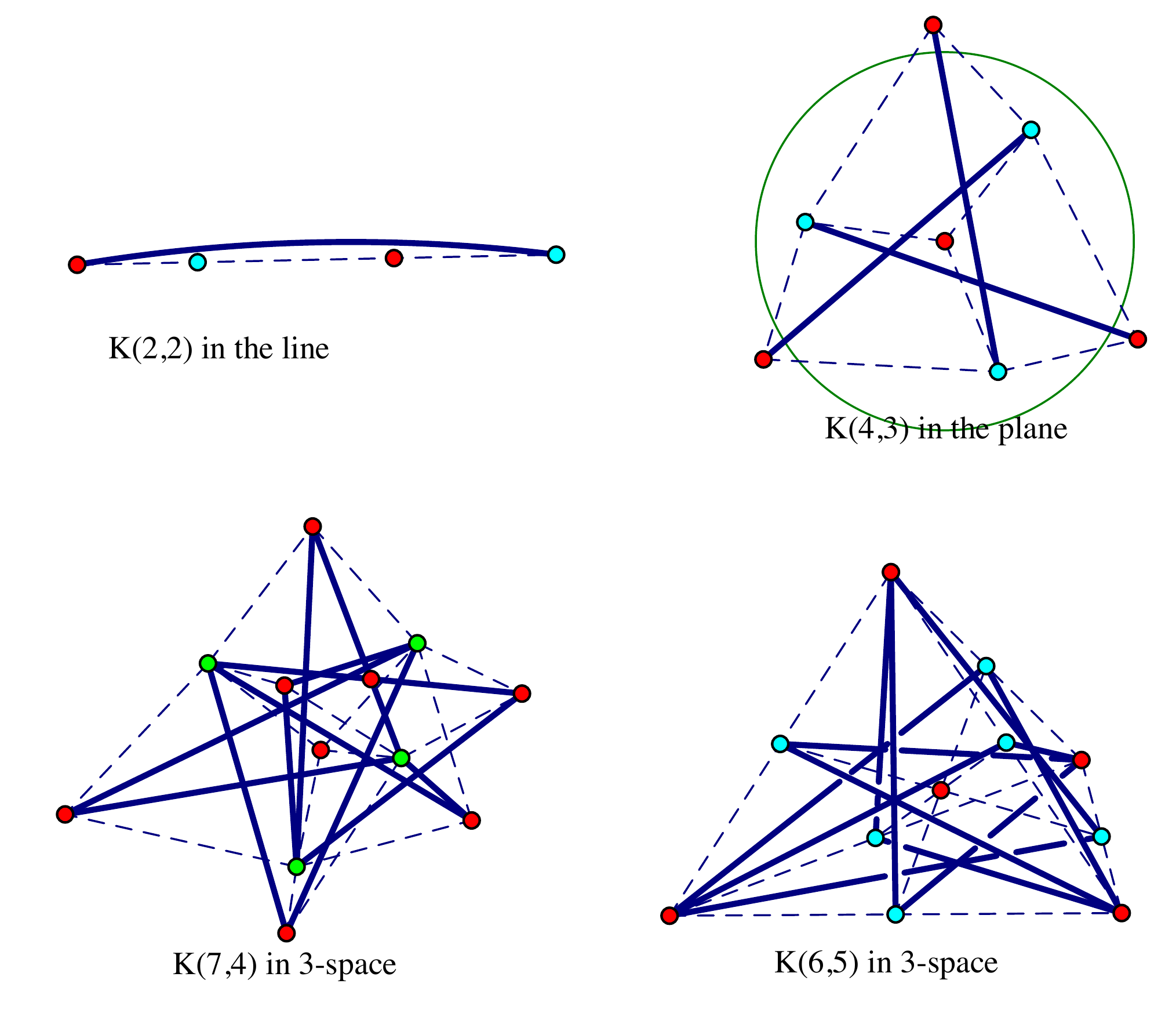}
        \end{center}
        \caption{ }\label{fig:min}
    \end{figure}
 \clearpage
In Figure \ref{fig:K33} the top examples are frameworks of the graph
$K(3,3)$ in the plane.  The top left example is super stable.  It lies
on a conic, which corresponds to a co-dimension two subspace of $\M_3$.
It has equilibrium stress which is PSD since it cannot be
separated by a conic.  See also \cite{connelly-energy}.  The top right
example is not universally rigid, even though the vertices lie on a
conic, since the partitions can be separated by a conic consisting of
two lines, as shown.  The bottom example is the same as the top left
example, except a red vertex is inserted and attached to the blue
vertices forming a $K(4,3)$. The stress on the edges on the central
vertex is zero, but the entire configuration is in general position in
the plane.  So Corollary \ref{cor:general-position} applies and it is
universally rigid, but not super stable.
\begin{figure}[here]
    \begin{center}
        \includegraphics[width=0.8\textwidth]{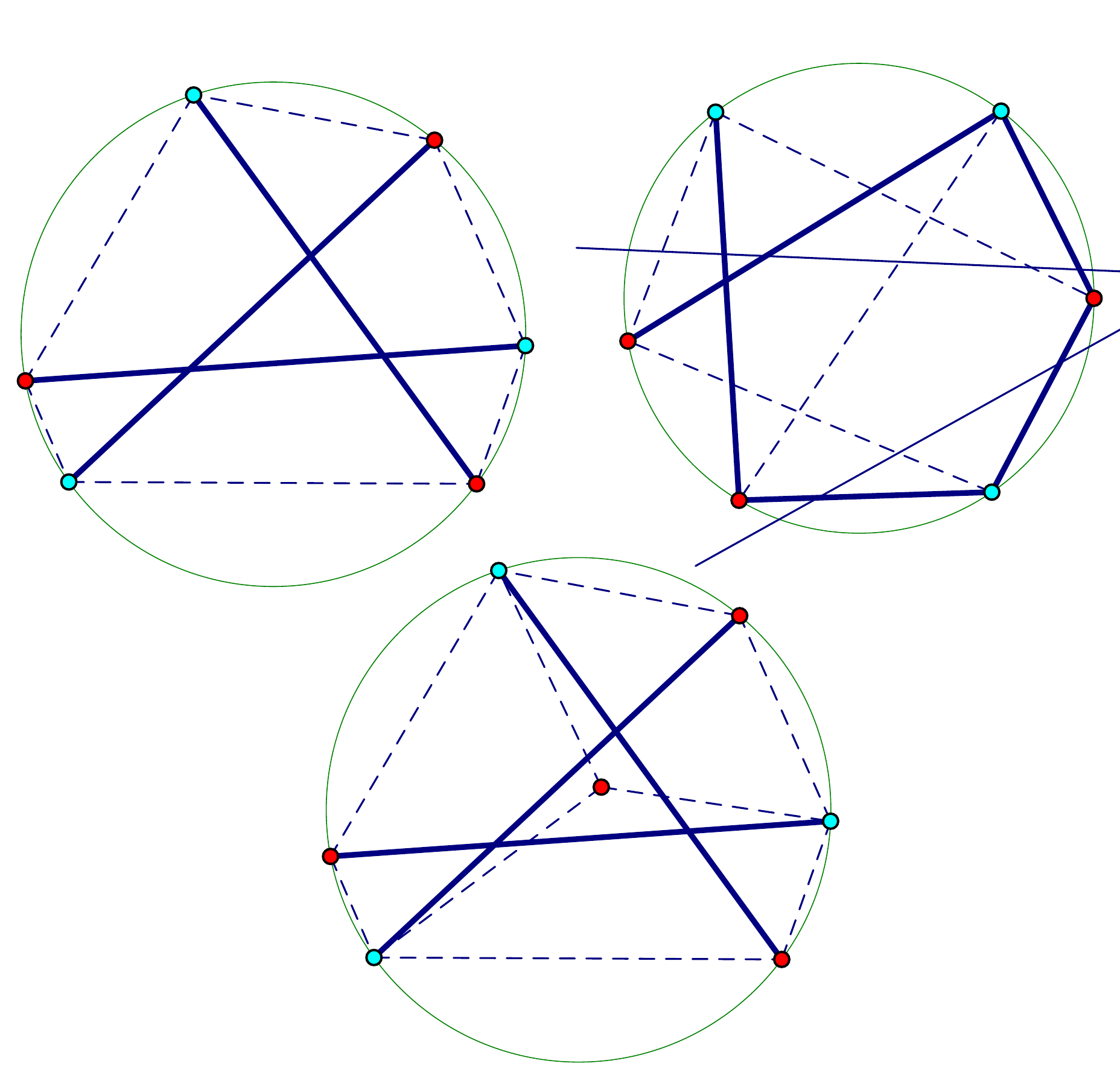}
        \end{center}
        \caption{ }\label{fig:K33}
    \end{figure}
    
\clearpage 
    
\section{Primitive Cores } \label{section:cores}

\begin{definition}Following \cite{Eckhoff}, (Theorem 9.1) we say that a partition\\ $(\V(\p), \V(\q))$ in $\M^d$ is \emph{primitive} if the convex hull of $\V(\p)$ intersects the convex hull of  $\V(\q)$ and no proper subset of $(\V(\p), \V(\q))$ has this property.  
\end{definition}

From our discussion above and \cite{Eckhoff} it is clear that if the convex hull of  $\V(\p)$ intersects the convex hull of $\V(\q)$, there are subsets $\p' \subset \p$ and $\q' \subset \q$ such that the convex hull of $\V(\p')$ intersects the convex hull of  $\V(\q')$ in their relative interiors with a minimal number of vertices.  We call the subframework $(K(n',m'), \p',\q')$ a \emph{primitive core} of  $(K(n,m), \p,\q)$.   
Here we list all the primitive cores of complete bipartite graphs for dimensions one, two, three.  It is easy to see how to extend this higher dimensions.  

Note that when $K(n,m)$ is a primitive core with affine span of
dimension $d$, then
$n \ge d+1$,  $m \ge d+1$, and 
$n+m-2$ is the dimension of the affine span of
$(\V(\p), \V(\q))$ in $\M^d$.
Since
$(d+1)(d+2)/2-1$ is the dimension of the affine span of the image of
$\V(\R^d)$ in $\M_d$, the vertices of $(\p,\q)$ lie in the
intersection of $(d+1)(d+2)/2+1-(n+m)$ quadrics, corresponding to
hyperplanes in $\M_d$.  Furthermore, for a primitive core,
$(K(n,m),\p,\q)$ is super stable.

\subsection{Dimension one}\label{subsection:core-one}

There is only one primitive core given by $K(2,2)$, where the partitions alternate along the line, as in Figure \ref{fig:min}. This is the main result of \cite{Jordan-universal-line}.

\subsection{Dimension two}\label{subsection:core-two}

When the core vertices are in quadric general position there is only $K(4,3)$ as in Figure \ref{fig:min}.  Here the dimension of the affine span of $(\V(\p), \V(\q))$ is $5$-dimensional.

When the Veronese images of the 
core vertices $(\V(\p), \V(\q))$ have a $4$-dimensional affine span, then there is one more example, $K(3,3)$ as in Figure~\ref{fig:K33}.  The vertices of this $K(3,3)$ lie on a single conic in the plane.  This particular example was described in \cite{connelly-energy} as well, for example.

\subsection{Dimension three}\label{subsection:core-three}

When the core vertices are in quadric general position, there are two examples, $K(6,5)$ and $K(7,4)$ as in Figure \ref{fig:min}.  Here the the core vertices $(\V(\p), \V(\q))$ have a $9$-dimensional affine span in $\M_3$.

When the core vertices  $(\V(\p), \V(\q))$ have an $8$-dimensional affine span in $\M_3$, then there are two examples, $K(5,5)$ and $K(6,4)$.  Figure \ref{fig:K456} shows an example for $K(6,4)$ and $K(5,5)$ lying on a sphere.  The configuration for $K(6,4)$ is obtained by taking the green vertices as the vertices of a regular tetrahedron, and the red vertices as the midpoints of the $6$ edges rescaled out to be on the circumsphere of the tetrahedron.  The configuration for $K(5,5)$ is obtained by taking the red and green vertices as a regular octahedron, but with the red vertices translated up and the green vertices translated down.  Then another red and green vertex is added down and up, respectively, to avoid separating the two partitions.
\begin{figure}[here]
    \begin{center}
        \includegraphics[width=0.6\textwidth]{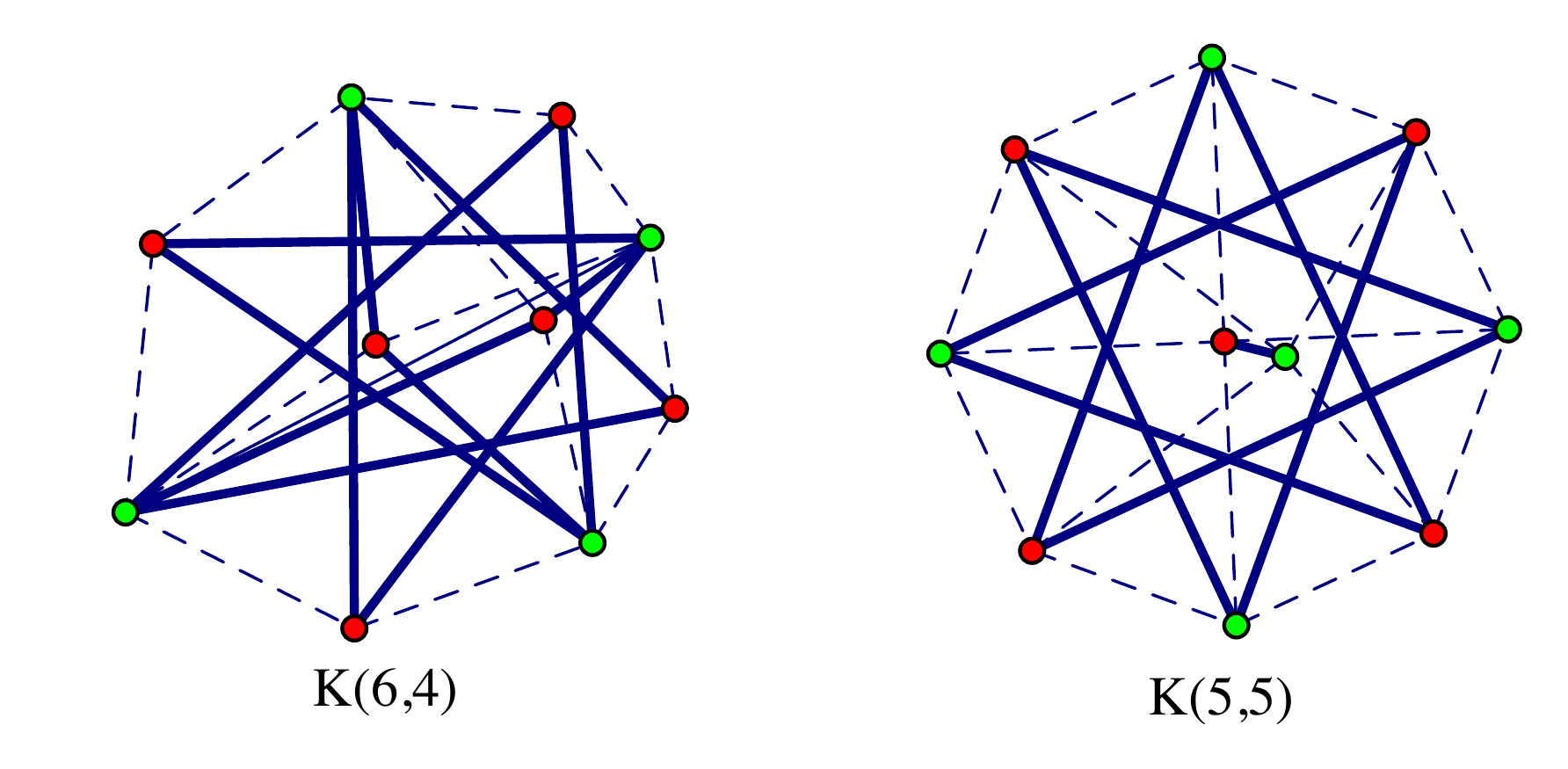}
        \end{center}
        \caption{ }\label{fig:K456}
    \end{figure}

When the core vertices  $(\V(\p), \V(\q))$ have a $7$-dimensional affine span in $\M_3$, then there is an example, $K(5,4)$.  (One can use the analysis of Theorem \ref{thm:SVD} to construct examples in this range.)  This configuration lies on the intersection of two quadrics.

When the core vertices $(\V(\p), \V(\q))$ have a $6$-dimensional affine span in $\M_3$,  then there is an example, $K(4,4)$, which is the intersection of three quadrics.  One example is a cube with its long diagonal as in Figure \ref{fig:cube}.  This was also shown in \cite{connelly-energy}.
\begin{figure}[here]
    \begin{center}
        \includegraphics[width=0.4\textwidth]{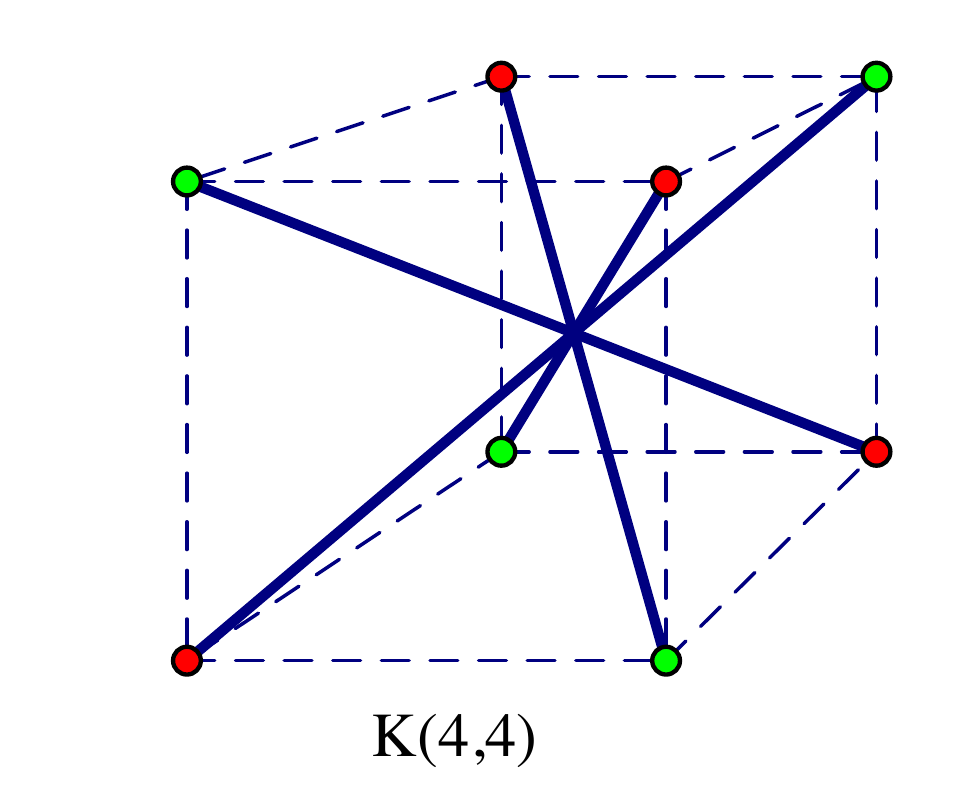}
        \end{center}
        \caption{ }\label{fig:cube}
    \end{figure}

\section{Coning and Projection-Section } \label{section:coning-projection}

Here we describe some general tools that are interesting in their own right
and that 
we will use below in 
Section~\ref{section:algorithm}.  See also \cite{Connelly-Whiteley-coning} for similar results in the context of generic global rigidity.

\subsection{Coning}\label{subsection:coning}

\begin{definition} A \emph{coned graph} is one where one of the vertices is connected to all the others.
\end{definition}

If a configuration for a complete bipartite graph has coincident
vertices from different partitions, we can identify those two vertices
as one, and we effectively have a coned graph. Here we first consider
a general graph, not just a bipartite graph, that has a distinguished
vertex $\p_0$ that is connected to all the vertices of a graph $G$.
We denote this framework as $\p_0*(G,\p)$.  We also assume that all
the vertices of $\p$ are distinct from $\p_0$.
The following is immediate since one can slide the
vertices of $G$ on the lines from $\p_0$ while preserving universal
and dimensional rigidity.
\begin{lemma}\label{lemma:slide}Suppose that $\p_0*(G,\p)$ and $\p_0*(G,\q)$ are two coned frameworks.
For simplicity, we assume $\p_0=0$,
the origin.  
Suppose  that for each edge $\{i,j\}$ of $G$, 
\begin{equation*}
\frac{|\p_i\cdot \p_j|}{|\p_i| |\p_j|} =\frac{|\q_i\cdot \q_j|}{|\q_i| |\q_j|}.
\end{equation*}
Then $\p_0*(G,\p)$ is universally rigid if and only if  $\p_0*(G,\q)$ is universally rigid, and $\p_0*(G,\p)$ is dimensionally rigid if and only if  $\p_0*(G,\q)$ is dimensionally rigid.
\end{lemma}
Figure \ref{fig:slide} shows this for a quadrilateral in the plane
that is a cone over a $K(2,2)$ graph on a line.  The
cable-strut designation is shown as well.  The stress values on the 
``cone edges''
over the collinear $K(2,2)$ are zero.
\begin{figure}[here]
    \begin{center}
        \includegraphics[width=0.5\textwidth]{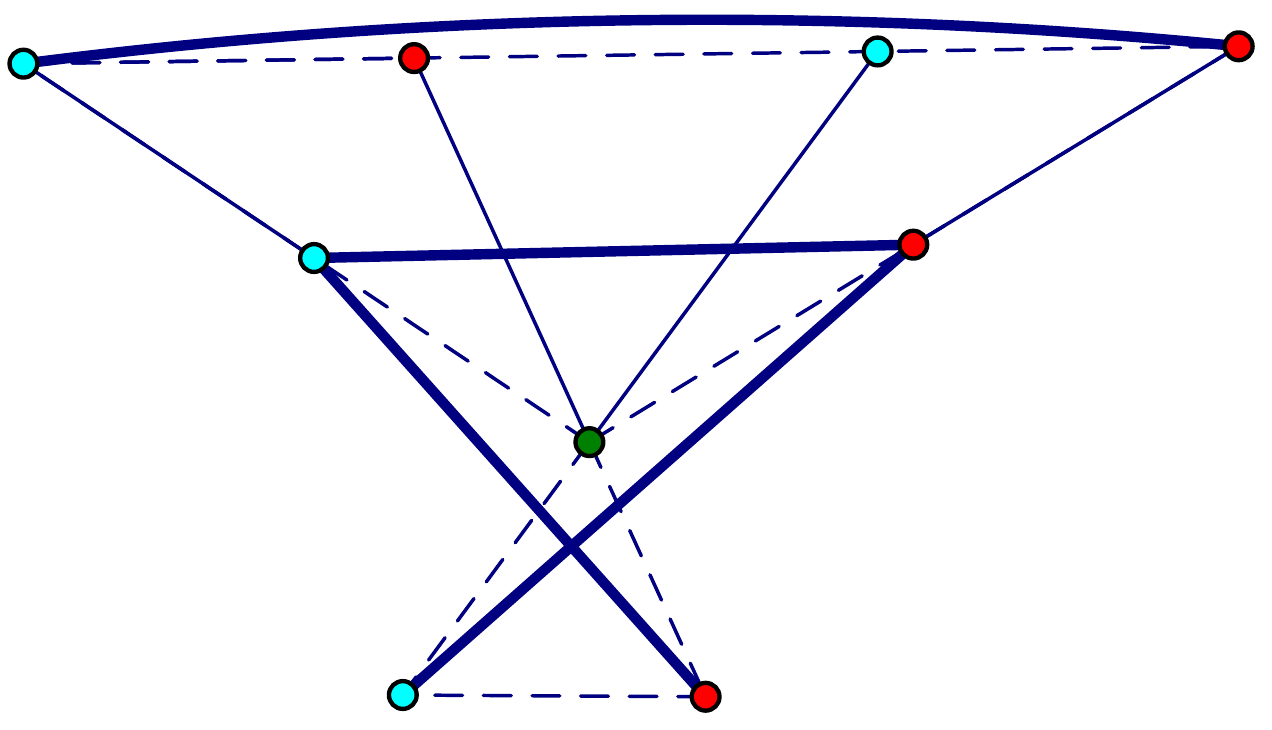}
        \end{center}
        \caption{ }\label{fig:slide}
    \end{figure}

The following is a general result relating universal and dimensional rigidity to their coned frameworks.
\begin{proposition}\label{prop:cone} 
Suppose that 
$\p$ is in
$\R^d$
and the cone point $\p_0 \in \R^{d+1} - \R^d$.
Then the framework 
$\p_0*(G,\p)$ is dimensionally rigid 
if and only if $(G,\p)$ is dimensionally rigid,
and if $(G,\p)$ is universally rigid, then $\p_0*(G,\p)$ is
universally rigid.
\end{proposition}
\begin{proof}The  ``if" statements are obvious.

For the other direction, Suppose $(G,\p)$ is 
not dimensionally rigid. 
For now, we will assume that $G$ is connected.
Without loss of generality,
we can choose $d$ so that the span of $\p$ is full within 
$\R^d$.

In $\R^{d+1}$ construct a parallel framework $(G,\q)$
by translating each vertex $\p_i$ by one unit perpendicular to the
$\R^{d}$ hyperplane to get $\q_i$.  Then for each edge $\{i,j\}$ of
$G$, construct the bars connecting all the pairs of vertices $\p_i,
\p_j,\q_i, \q_j$ constructing a new framework in $\R^{d+1}$, $(H,
(\p,\q))$, as in Figure~\ref{fig:parallel}.
\begin{figure}[here]
    \begin{center}
        \includegraphics[width=0.5\textwidth]{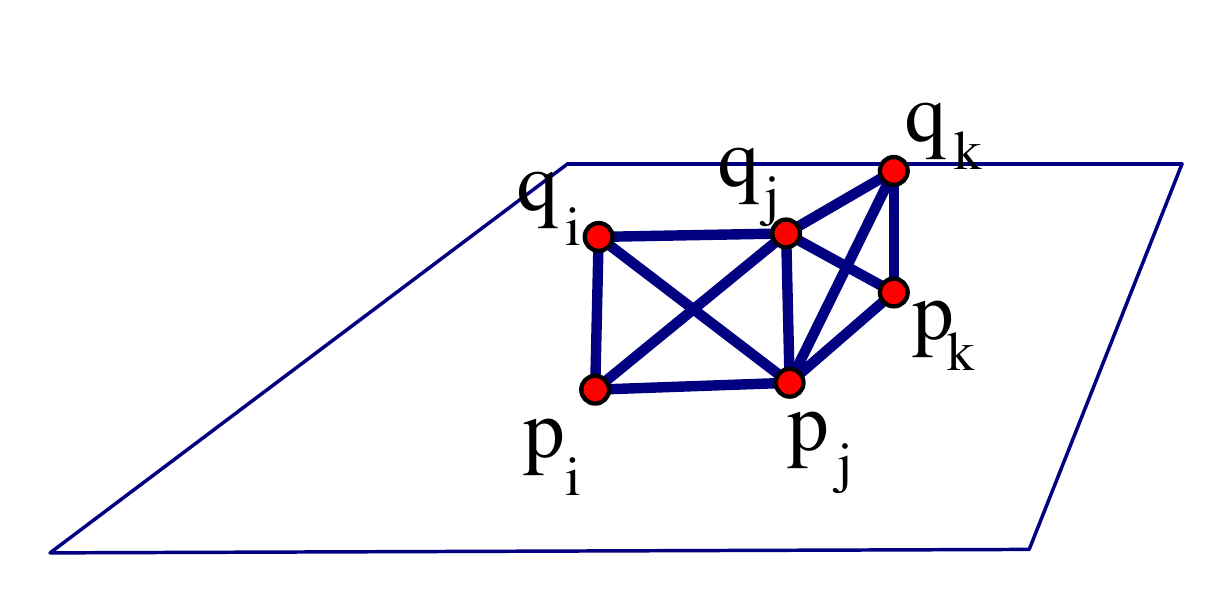}
        \end{center}
        \caption{ }\label{fig:parallel}
    \end{figure}
    
It is clear that 
if  $(G,\p)$ is 
not dimensionally rigid 
then
$(H, (\p,\q))$ is 
not dimensionally rigid.  Then by Section
13 of \cite{Connelly-Gortler-iterated}, any non-singular projective
image of $(H, (\p,\q))$,
is 
not dimensionally rigid as well.
But the lines through $\p_i$ and $\q_i$ are all parallel and so in the
projective image all these lines intersect at 
a ``meeting''  point in
$\R^{d+1}$.

For any chosen point $\p_0$, we can find a projective transformation
that leaves $\R^d$ fixed and such that the image of $(H, (\p,\q))$
has its
meeting point at $\p_0$. Let us 
denote this framework as
$(H, (\p,\q'))$.
The point $\p_0$ is on each of
the lines through $\p_i, \q'_i$ for all $i$.

Each edge $\{i,j\}$ of $G$ corresponds to a $4$-vertex universally
rigid planar  framework on the vertices $\p_i, \p_j,
\q'_i, \q'_j$ in the graph $H$. 
Each such $4$-vertex framework
determines a unique ``apex point'',
say, using 
the angle-side-angle theorem in elementary geometry.
This also determines the distance from the apex point to $\p_i$ and
to $\q'_i$ along the line spanned by  $\p_i$ and $\q'_i$.
See
Figure~\ref{fig:ASA}.
In the framework,
$(H, (\p,\q'))$ all of these
apicies coincide at $\p_0$.

 \begin{figure}[here]
    \begin{center}
        \includegraphics[width=0.3\textwidth]{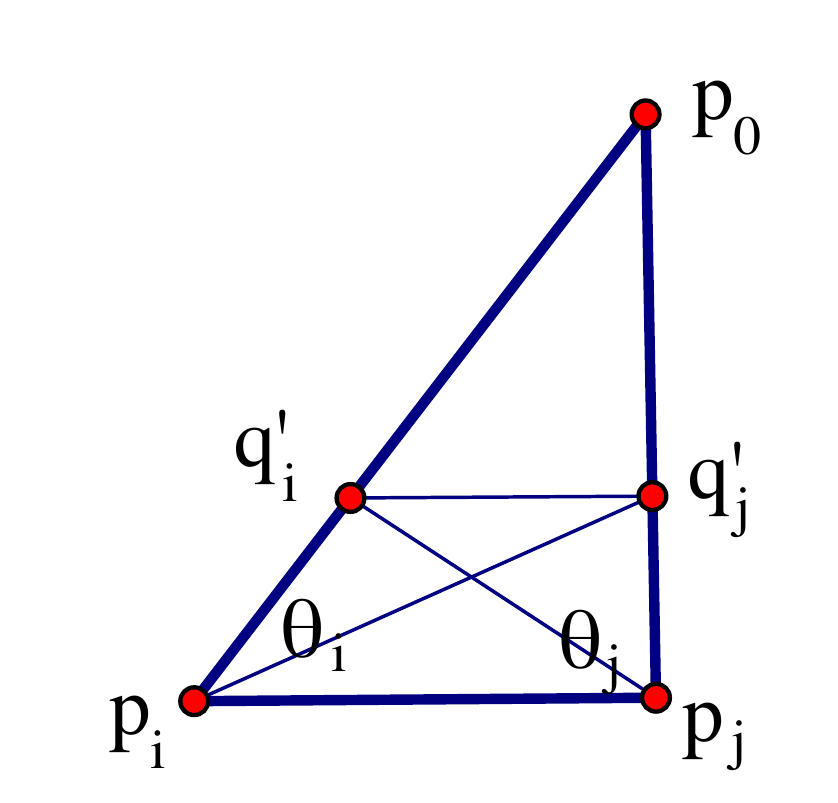}
        \end{center}
        \caption{ }\label{fig:ASA}
    \end{figure}

Suppose there is 
a  second framework of $H$ with the same edge lengths as
$(H, (\p,\q'))$ but with an affine span of dimension greater than
$d+1$. Then for each $4$-vertex set
in this second framework, we can
compute its apex point. Since we have assumed that $G$ is connected,
these all must agree at a common meeting point.
This means that we can find a second framework that has the same edge lengths
as the coned framework
$\p_0*(H,(\p,\q'))$, but with an affine span greater than $d+1$.

From  Lemma \ref{lemma:slide}, 
this means that we can find 
a second framework that has the same edge lengths
as the coned framework
$\p_0*(G,\p)$  
but with an affine span greater than $d+1,$
making it not dimensionally rigid.

Finally, we can look at the case that $(G,\p)$ in $\R^d$ has multiple connected
components. 
Suppose one of the components is not dimensionally rigid.
We have just shown that
the coned framework over that component
is not dimensionally rigid as well.
This can be used to certify that
$\p_0*(G,\p)$  is not dimensionally rigid.
Suppose, instead that 
each of the components is dimensionally rigid 
but $(G,\p)$ is not. This means that we can 
increase the dimension span by simply rigidly moving one of the
components into a larger space. The same will be true for
$\p_0*(G,\p)$ using an appropriate rotation of that component about $\p_0$
into some larger space.
$\Box$
\end{proof}

In the setting of the above proposition,
it is not true that if $\p_0*(G,\p)$ is universally rigid,
then $(G,\p)$ is universally rigid.  In particular, the proof of the 
proposition above relies on the invariance of dimensional rigidity
with respect to projective transformations. This invariance does
not hold for universal rigidity as having edge directions 
on a conic at infinity is not invariant with respect to projective
transformations!
Following
the example of Figure~8 in~\cite{Connelly-Gortler-iterated}, the
ladder in $\R^2$ as in Figure \ref{fig:ladders} is not universally
rigid, since it has an affine flex, but the cone over the ladder in
$\R^3$ is universally rigid, since it has a section, the orchard
ladder which is universally rigid.

  \begin{figure}[here]
    \begin{center}
        \includegraphics[width=0.5\textwidth]{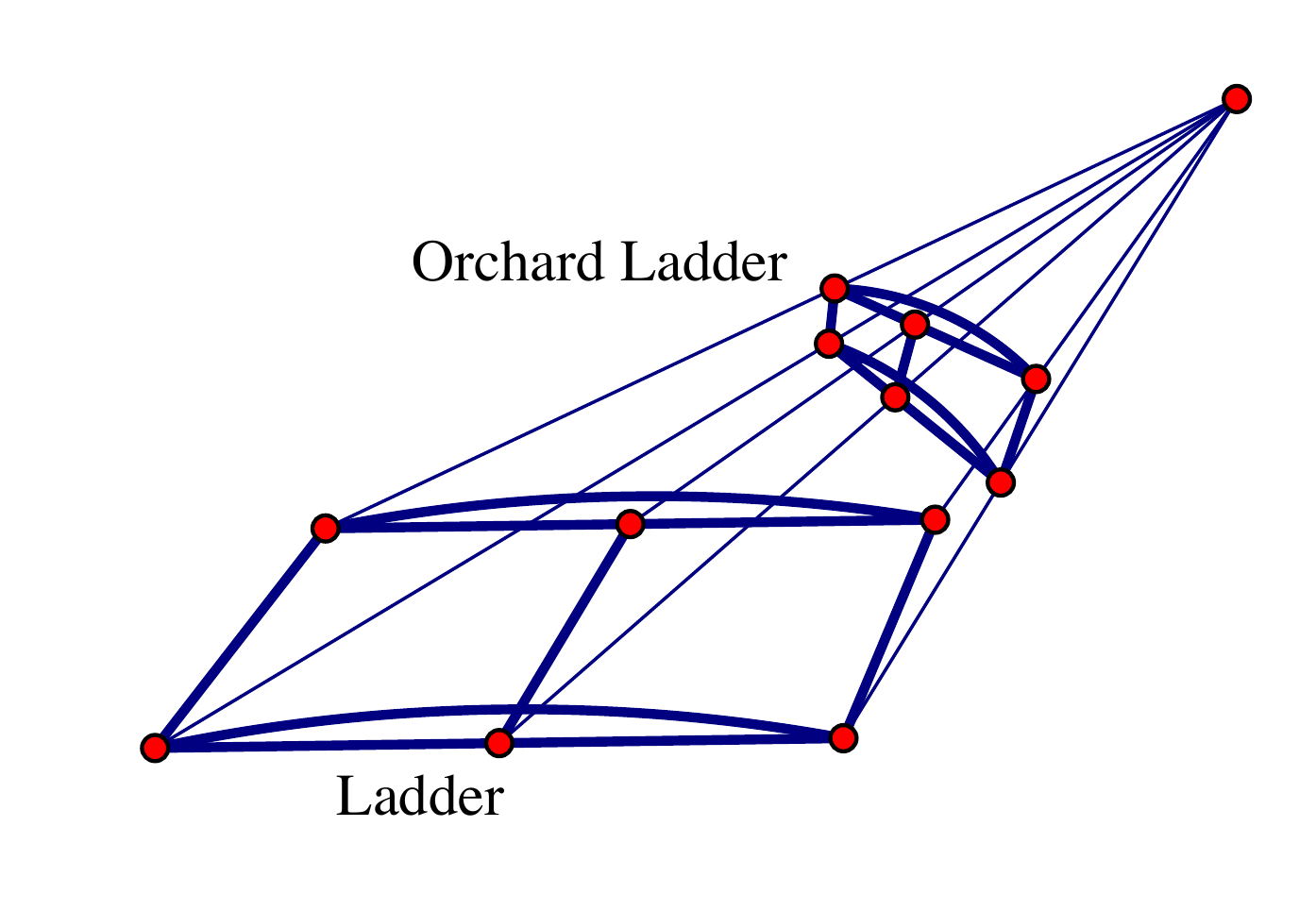}
        \end{center}
        \caption{ }\label{fig:ladders}
    \end{figure}
    
But when we specialize to complete bipartite graphs, examples such as
with Figure \ref{fig:ladders} can be ruled out, and we note
the following corollary (which we will not need elsewhere in this paper).
 
\begin{corollary}\label{corollary:cone} In Proposition \ref{prop:cone}, if we assume in addition that the graph $G=K(n,m)$ is a complete bipartite graph, each of $\p$ and $\q$ span $\R^d$, and \\
$\p_0*(K(n,m),(\p,\q))$ is universally rigid, then $(K(n,m),(\p,\q))$
  is universally rigid.
\end{corollary}
\begin{proof}  By  Proposition \ref{prop:cone} since $\p_0*(K(n,m),(\p,\q))$ is universally rigid, 
$(K(n,m),(\p,\q))$ is dimensionally rigid, and Lemma
  \ref{lemma:proper-affine} implies that $(K(n,m),(\p,\q))$ is
  universally rigid.  $\Box$
\end{proof}

\subsection{Projections and Cross-sections}\label{subsection:sections}

Suppose that 
$\tilde{{V}} \subset {V}$
is a subset of the vertices of a graph $G$, which induces a subgraph
$\tilde{G}$
where the edges 
$E(\tilde{G}) \subset E(G)$.
$\tilde{{V}}$ 
thus induces 
$\tilde{\p}$, a subconfiguration of the points
$\p$.
This gives us 
$(\tilde{G}, \tilde{\p})$, a subframework of $(G, \p)$.
\begin{lemma}\label{lemma:projection} 
Suppose that $(\tilde{G}, \tilde{\p})$ is a universally rigid
subframework of $(G, \p)$ in $\R^d$, 
where 
the dimension of the affine span of
$\p$ is $d$, 
and
the dimension of the affine span of
$\tilde{\p}$ is $\tilde{d} < d$.
Suppose
further that for each vertex not in $\tilde{{V}}$,
the dimension of the affine
span of its 
neighbors in $\tilde{\p}$ is $\tilde{d}$-dimensional.  
Let
$\pi : \R^d \rightarrow \R^{d-\tilde{d}}$ be the 
orthogonal projection
that projects all the points
of $\tilde{\p}$ to a single point, say $\p_0$.

Then $(G, \p)$ is universally rigid (respectively dimensionally rigid)
if and only if $\p_0*(G,\pi(\p))$ is universally rigid (respectively
dimensionally rigid).
\end{lemma}  
\begin{proof}  
We are regarding $\R^d=\R^{\tilde{d}} \times \R^{d-\tilde{d}}$ such that 
$\tilde{\p}
\subset \R^{\tilde{d}}$.  Since, for $\p_i$ corresponding to a vertex of 
${V}-\tilde{{V}}$,
the dimension
of the affine span of its neighbors  in $\tilde{\p}$ is
$\tilde{d}$-dimensional, then
the distance from such $\p_i$ to
$\R^{\tilde{d}}$ is constant for any equivalent realization of $(G, \p)$
fixing $\tilde{\p}$ and is equal to $|\p_0-\pi(\p_i)|$. Additionally with
$\tilde{\p}$ fixed, $\tilde{\pi}(\p_i)$ is fixed as well, where 
$\tilde{\pi} : \R^d
\rightarrow \R^{\tilde{d}}$ is the orthogonal
projection onto $\R^{\tilde{d}}$.  Similarly, for
$\{i,j\}$ an edge of 
$G-\tilde{G}$, $|\pi(\p_i)-\pi(\p_j)|^2 +
|\tilde{\pi}(\p_i)-\tilde{\pi}(\p_j)|^2= |\p_i-\p_j|^2$.  Then the conclusion
follows.  $\Box$
\end{proof}

Figure \ref{fig:projection} 
shows an example of a universally rigid $K(4,4)$ in $\R^3$, using Lemma \ref{lemma:projection} applied  to $K(2,2)$ on a line in $\R^3$, then Lemma \ref{lemma:slide} and  Corollary~\ref{corollary:cone} applied to another $K(2,2)$ on a line, this time in $\R^2$. 
\begin{figure}[here]
    \begin{center}
        \includegraphics[width=0.5\textwidth]{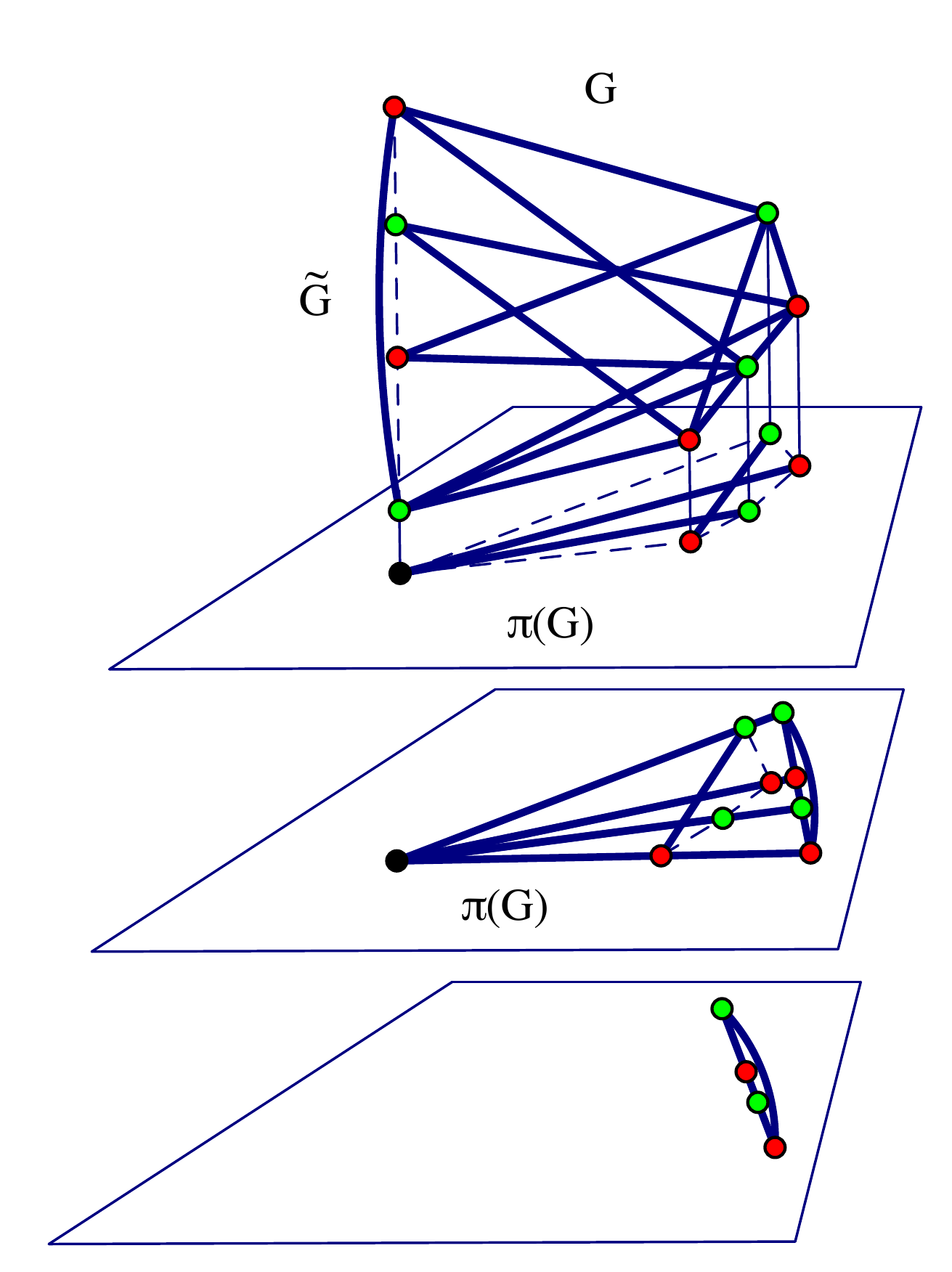}
        \end{center}
        \caption{ }\label{fig:projection}
    \end{figure}

\section{Algorithm}
\label{section:algorithm}

We can completely test dimensional and universal
rigidity of any complete bipartite framework
with an efficient algorithm which we describe now. 
We will assume that the input coordinates $(\p,\q)$ in $\R^d$ are 
given as rational numbers
that can be described with $L$ bits. 
Without loss of generality, we will assume that the
affine span of $(\p,\q)$ is $d$-dimensional.

Though it is true that one can also attempt
to numerically gain evidence to 
answer this question using semidefinite programming~\cite{Ye-So-sensor},
the lack of complexity results for SDP feasibility~\cite{ramana}
makes that approach theoretically less satisfying.

At the heart of our algorithm is a routine that 
looks for a solution to the following set of conditions
over the variables $\lambda_i, \mu_j$:
\ba
\sum_{i=1}^n \lambda_i \hat{\p}_i \hat{\p}_i^t
&=& 
\sum_{i=1}^m \mu_j \hat{\q}_j \hat{\q}_j^t \\
\sum_{i=1}^n \lambda_i &=& 1 \\
\lambda_i &\geq& 0 \\
\mu_j &\geq& 0 
\ea
The second condition rules out the all-zero solution.
This is a linear programming feasibility problem that can be 
exactly solved
in worst case time that is polynomial in $(n+m,L)$.

Let us, for now, assume that 
$m+n \geq {\rm dimSpan}(\p,\q)+2$.
If there is no feasible solution, then from
Theorem~\ref{thm:main}, the graph must be 
not dimensionally rigid.
On the other hand
if we find a feasible solution and \emph{all} of the
$\lambda_i$ and $\mu_j$  have positive values, then we 
know that there is a
maximum rank PSD equilibrium stress matrix on the complete bipartite
framework and
from Theorem~\ref{thm:SVD},
our framework must
be super stable, and thus universally rigid.

Suppose though,
we find a feasible solution, where some, but not all of the 
$\lambda_i$ and $\mu_j$ have positive values.
We can easily determine which
$\lambda_i$ and $\mu_j$  have positive values, and we 
will know that there is a
maximum rank PSD equilibrium stress matrix on the complete bipartite
subframework over the associated $\p_i$ and $\q_j$. 
From Theorem~\ref{thm:SVD},
this subframework must
be super stable and universally rigid.
But what can we say about the complete input framework?

For example, Figure \ref{fig:flexible}
shows a $2$-dimensional framework, 
where our linear program will find a certifying PSD stress for the
collinear $K(2,2)$ subframework.
But the vertices that are not on this common line
are  free to flex continuously in three dimensions.

\begin{figure}[here]
    \begin{center}
        \includegraphics[width=0.3\textwidth]{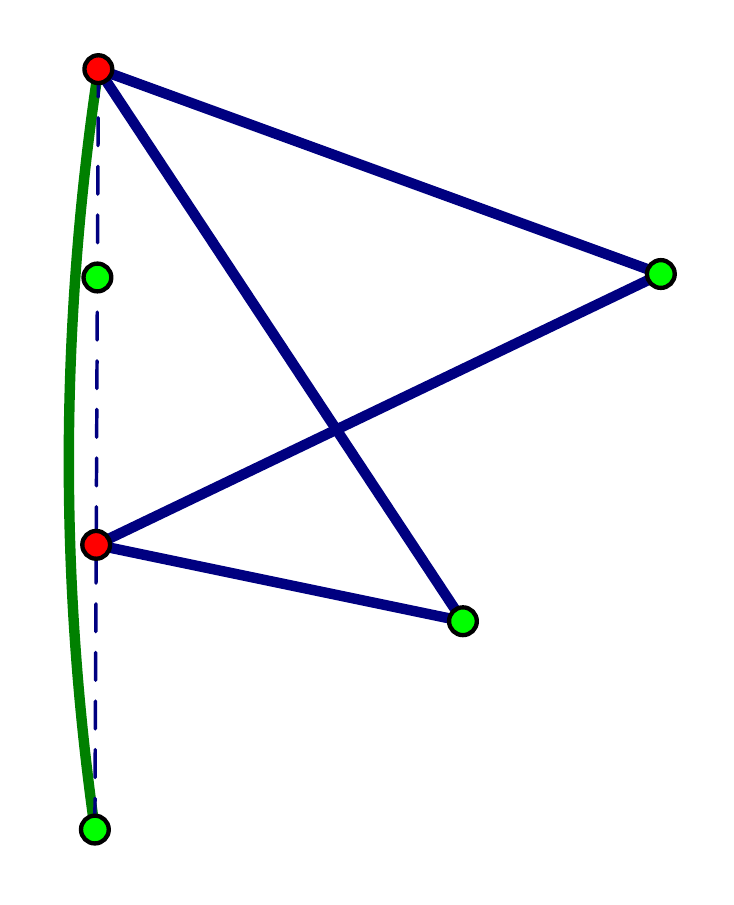}
        \end{center}
        \caption{ }\label{fig:flexible}
    \end{figure}

The idea of our algorithm is to proceed by recording the 
indices of our ``already known to be universally rigid subframework'',
and then to apply ideas from Section
\ref{section:coning-projection}, to reduce our problem down to 
a smaller problem.
Roughly speaking, we will 
project the known universally subframework down to a cone point,
slide the remaining points into a common hyperplane, and 
finally  remove
the cone point
(see e.g. Figure \ref{fig:projection}).
We can we can then apply our linear programming approach to
this smaller problem. Infeasibility of the smaller problem will
imply that the smaller problem is 
not dimensionally rigid and so too is
the
original framework,
and we can exit.
A feasible solution for the smaller problem
will allow us to add even more vertices 
to the known universally rigid set.
We iterate this process  until we either exit due to infeasiblity,  
or we account for all of the
vertices and conclude that our input framework is universally rigid, 
or we end up with a smaller problem where the number of vertices is 
exactly one more than the dimension of their affine span, in which case
the smaller problem, as well as the original framework is 
dimensionally but not universally rigid.

We note that after the first stage, 
due to the geometric projection and sliding operations
(both which can be determined using a linear system)
the input to our linear program
may require polynomially more bits than our original input size, $L$.
But in each stage of the iteration, we perform these geometric
operations anew starting with the input $(\p,\q)$ data, so 
this avoids any cascading blowup in bit complexity.

\subsection{Details}

In this algorithm, let  $\tilde{{V}} \subset {V}$
be an index set recording 
the vertices of 
some complete bipartite subframework of $(K(n,m),\p,\q)$ in $\R^d$, 
which has 
already been determined to be 
universally rigid. We will refer to this subframework as the
``known-UR set''.
The known-UR set begins as empty.
During the algorithm, the known-UR set  will also maintain the 
``invariant'' property that
the affine span of its $\p$-subset agrees with that of its
$\q$-subset.

The complement of the known-UR set is denoted as
${{V}'}$, describing a complete bipartite subgraph
$K(n',m')$.
Suppose
the complement is empty, then the known-UR set is the entire 
$(\p,\q)$, and thus $(K(n,m),\p,\q)$ is universally rigid. The same is true
(due to the invariant) if  the complement consists of 
a single $\p$, or a single $\q$, or one $\p$ and one $\q$
(which 
must be connected by an edge).

Our algorithm is the following:
\\ \\
\noindent
RigidityTest($\p,\q)$ \\
\hspace*{.4 in}
$\tilde{{V}}$ := $\{\}$ \\
\hspace*{.4 in}
repeat \\
\hspace*{.6 in}
${{V}'}$ 
:= ${V} - \tilde{{V}}$ \\
\hspace*{.6 in}
if (
$\#({{V}'}_{\p}) 
\leq 1$ 
and 
$\#({{V}'}_{\q}) 
\leq 1$)
\hspace*{.4 in}
  output ``universally rigid''\\
\hspace*{.6 in}
$(\p_0,\p',\q')$ := $\pi_{\tilde{{V}}} $($\p,\q$) \\
\hspace*{.6 in}
$(\p'',\q'')$ := $\sigma_{\p_0}$($\p',\q'$) \\
\hspace*{.6 in}
if (dimspan($\p'',\q''$) = ($\#({{V}'} )-1$))
\hspace*{.25 in}  output ``dimensionally rigid''\\
\hspace*{.6 in}
${{S}}$
:= findSuperStableSubframework($\p'',\q'',{{V}'}$) \\
\hspace*{.6 in}
if ($\#({{S}}) = 0)$
\hspace*{1.42 in}
output ``not dimensionally rigid''\\
\hspace*{.6 in}
$\tilde{{V}}$
:= 
$\tilde{{V}} \cup {{S}}$ \\
\hspace*{.6 in}
$\tilde{{V}}$
:= affineClosure$_{\p,\q}$($\tilde{{V}}$)
\\

Let us denote the dimension of  the known-UR set as
$\tilde{d}$. 
The function
$\pi_{\tilde{{V}}}(\p,\q)$ 
performs the orthogonal projection
on  the points $(\p,\q)$ 
such that  the vertices of the known-UR set project to a single
point in $\R^{d-\tilde{d}}$. 
We denote this single point as $\p_0$, and the projection of the complementary
vertices as $(\p',\q')$.
We can think of this result as describing 
a framework 
$\p_0*(K(n',m'),\p',\q')$
of a cone over the complementary complete bipartite graph
in $\R^{d-\tilde{d}}$. 
(See Figure \ref{fig:projection}, top). 

The function
$\sigma_{\p_0}(\p',\q')$ 
 slides  the points in $\p'$ and $\q'$  
along their rays from  the cone point $\p_0$
such 
that they all lie in a hyperplane that does not include the cone vertex. 
(See Figure \ref{fig:projection}, middle). 
We denote the
resulting points 
as $(\p'',\q'')$.
By discarding the  cone point, 
we can think of this result as describing 
a framework 
$(K(n',m'),\p'',\q'')$
of the complementary complete bipartite graph
in $\R^{d-\tilde{d}-1}$. 
(See Figure \ref{fig:projection}, bottom).

Suppose 
that 
$(\p'',\q'')$
 has an affine span of
maximal dimension, one less then the total number of its vertices.
This, and the fact that the complementary  graph is not a simplex,
makes $(K(n',m'),\p'',\q'')$
 dimensionally but not universally rigid.
(This follows from simply counting the number of degrees of freedom
vs. constraints, 
the fact that the framework cannot have
any non-zero equilibrium stress, and an application of the 
main results of~\cite{Asimow-Roth-I}.)

Likewise $\p_0*(K(n',m'),\p'',\q'')$, being coned in one higher
dimension, is also of maximal dimension and not a simplex,
making it
dimensionally but not universally
rigid. 
Since  $(\p',\q')$ is obtained from 
$(\p'',\q'')$ using sliding through $\p_0$, then by 
Lemma \ref{lemma:slide}, $\p_0*(K(n',m'),\p',\q')$
too 
is dimensionally but not universally rigid. 
By Lemma \ref{lemma:projection},  $(K(n,m),\p,\q)$ is 
dimensionally but not universally rigid. 
Thus we output ``dimensionally rigid''.

The next step,
findSuperStableSubframework,
is the heart of the algorithm.
Here we find a subframework 
of 
$(K(n',m'),\p'',\q'')$ 
such that Equation (\ref{eqn:Radon1}) holds with strictly positive coefficients.
As described above,
this can be found by setting up an
a linear programming feasibility problem.
The output of this step is simply the indices of the vertices,
${S} \subset {{V}'}$
comprising this super stable subframework. 

If ${S}$  is 
empty, then from Theorem \ref{thm:Alfakih-stress} 
$(K(n',m'),\p'',\q'')$ is not dimensionally rigid. 
Then by 
Proposition \ref{prop:cone}, 
so too is 
$\p_0*(K(n',m'),\p'',\q'')$, 
then by 
Lemma \ref{lemma:slide}, 
so too is 
$\p_0*(K(n',m'),\p',\q')$, 
then by 
Lemma \ref{lemma:projection}, 
so too is 
$(K(n,m),\p,\q)$.
Thus we output ``not dimensionally rigid''.

If ${S}$  is not
empty, 
then from Theorem \ref{thm:SVD}, 
the subframwork of \\
$(K(n',m'),\p'',\q'')$ induced by 
the vertices in ${S}$
is universally rigid. As before, so is the induced subframeworks of 
$\p_0*(K(n',m'),\p'',\q'')$ and\\ 
$\p_0*(K(n',m'),\p',\q')$.  Finally, by Lemma \ref{lemma:projection}, and our invariant,
so is 
the subframework of $(K(n,m),\p,\q)$ 
induced by 
$\tilde{{V}} \cup {{S}}$.

Likewise, from Theorem \ref{thm:SVD}, 
the affine span  of the ${S}$-induced subset of 
$\p''$ agrees with that of 
the ${S}$-induced subset of 
$\q''$.
Thus the invariant
is also true for 
the subset of $(\p,\q)$ 
induced by 
$\tilde{{V}} \cup {{S}}$.
Thus we now include these vertices in our updated
known-UR set.

Finally using a trilateration argument, we can also add to 
$\tilde{{V}}$ any other the vertices that are in its affine span.

Each iteration in this algorithm always makes progress so it
must terminate after at most $n+m$ steps.

In summary:
\begin{theorem}
Given a complete bipartite framework with rational coordinates.
There is a (weakly) polynomial time algorithm that determines whether 
or not the
framework is dimensionally rigid, and whether or not it is 
universally rigid.
\end{theorem}

Running this algorithm on the example of Figure \ref{fig:projection},
will conclude in two iterations, that the framework is universally rigid.
For the example of Figure \ref{fig:flexible}, in the second iteration, 
${{V}'}$ will consist of two green vertices in 
$\R^0$, making ${S}$ empty, and $(\p,\q)$ 
not dimensionally rigid.

\section{Tensegrities and Further Work}  

An important consequence of our approach in this paper is that quite often we can replace the distance equality constraints with inequality constraints as described in Section \ref{section:introduction}.   Each edge of the underlying graph $G$ is designated as a cable, strut or bar depending on whether it is constrained not to increase, not to decrease or not to change length, respectively.   In \cite{Connelly-Gortler-iterated} we have shown that, in many cases, even though the given framework may not support an equilibrium stress that is non-zero for a given edge, it may still be possible to declare a given edge a cable or strut and maintain universal rigidity.  Even in the case when the graph is not bipartite, Proposition \ref{prop:cone} can apply as in Figure \ref{fig:slide}, where due attention should be applied to the signs  of the stresses.  We do not pursue that extension of the results here, though. 

Another application of our approach here is in the local rigidity theory of prestress stability as shown in \cite{connelly-whiteley}.  There, even if the stress matrix is not PSD, it can still be useful determine local rigidity, especially when the whole framework is not infinitesimally rigid.  

Another point is that the stress-energy function determined by the stress and stress matrix provides a measure of how far a given configuration is from an ideal configuration, globally.  So if a configuration has some determined edge measurements, the stress-energy function gives an upper bound on how close any configuration is with those edge lengths.  Indeed, with the tensegrity constraints it can be possible to eliminate certain edge lengths as feasible.  For example, for six points $\p_1, \dots, \p_6$, there is no configuration where $|\p_i -\p_{i+1}| \le 1$, and $|\p_i -\p_{i+3}| > 2$, all taken modulo $6$.  This is shown using the configuration $K(3,3)$ on a circle as in Figure \ref{fig:K33}.

\section{Acknowledgements}

The impetus for this paper is the result in \cite{Jordan-universal-line} for 
$K(n,m)$ on a line.  It was a desire to generalize that result, which was the starting point for this paper. 

The elephant in the room is the paper by E.  Bolker and B. Roth
\cite{Bolker-Roth}.  This paper was constantly in the background
leading us to what was true and what was not.  It gives a reasonably
complete picture of which configurations of complete bipartite graphs
are infinitesimally rigid.  Also, one can see stress matrices there
quite naturally.  Their basic tool was the tensor product of a vector
with itself, where instead we think of it as using the Veronese
map. 

Other work we did not formally use, but is still lurking in the background, is the very insightful paper \cite{Whiteley-bipartite} by W. Whiteley.  The idea there is that an infinitesimal  flex $\p'$ of a bipartite framework with corresponding configuration $\p$ on a quadric can be easily described.  Furthermore, the two configuration $\p +\p'$ and  $\p -\p'$ describe equivalent frameworks. Thus they are not even globally rigid, and they are separated by a quadric surface.  This is the basis in \cite{Connelly-K33} to show that $K(5,5)$ is not globally rigid (thus not universally rigid) in $\R^3$.  But on the other hand, there are many examples of complete bipartite graphs in any   $\R^d$ that are globally rigid, but not universally, as we have shown here.

The main result of \cite{Gortler-Thurston2} applied to complete bipartite graphs, shows that when the configuration is generic, the rank and positive semi-definiteness of the stress matrix determines when the configuration is universally rigid.  What  we have done here, for complete bipartite graphs, is to replace the condition of being generic, which is problematic to determine in general, with the more precise condition of being in quadric general position in Corollary \ref{cor:quadric-general-position}.

We would like to thank Deborah Alves whose experiments kept on suggesting the correctness of Theorem \ref{thm:SVD}, long before we knew how to prove it.

\bibliographystyle{plain}
\bibliography{framework}

\begin{thebibliography}{10}

\bibitem{Alfakih-bar-frameworks}
Abdo~Y. Alfakih.
\newblock On bar frameworks, stress matrices and semidefinite programming.
\newblock {\em Math. Program.}, 129(1, Ser. B):113--128, 2011.

\bibitem{alfakih2014local}
Abdo~Y Alfakih.
\newblock Local, dimensional and universal rigidities: A unified gram matrix
  approach.
\newblock In {\em Rigidity and Symmetry}, pages 41--60. Springer, 2014.

\bibitem{Alfakih-Ye-general-position}
Abdo~Y. Alfakih and Yinyu Ye.
\newblock On affine motions and bar frameworks in general position.
\newblock {\em Linear Algebra Appl.}, 438(1):31--36, 2013.

\bibitem{Asimow-Roth-I}
Leonard Asimow and Ben Roth.
\newblock The rigidity of graphs.
\newblock {\em Transactions of the American Mathematical Society},
  245:279--289, 1978.

\bibitem{Bolker-Roth}
Ethan~D. Bolker and Ben Roth.
\newblock When is a bipartite graph a rigid framework?
\newblock {\em Pacific J. Math.}, 90(1):27--44, 1980.

\bibitem{connelly-energy}
Robert Connelly.
\newblock Rigidity and energy.
\newblock {\em Invent. Math.}, 66(1):11--33, 1982.

\bibitem{Connelly-K33}
Robert Connelly.
\newblock On generic global rigidity.
\newblock In {\em Applied geometry and discrete mathematics}, volume~4 of {\em
  DIMACS Ser. Discrete Math. Theoret. Comput. Sci.}, pages 147--155. Amer.
  Math. Soc., Providence, RI, 1991.

\bibitem{Connelly-Gortler-iterated}
Robert Connelly and Steven~J. Gortler.
\newblock Iterative universal rigidity.
\newblock {\em Discrete Comput. Geom.}, 53(4):847--877, 2015.

\bibitem{connelly-whiteley}
Robert Connelly and Walter Whiteley.
\newblock Second-order rigidity and prestress stability for tensegrity
  frameworks.
\newblock {\em SIAM J. Discrete Math.}, 9(3):453--491, 1996.

\bibitem{Connelly-Whiteley-coning}
Robert Connelly and Walter Whiteley.
\newblock Global rigidity: the effect of coning.
\newblock {\em Discrete Comput. Geom.}, 43(4):717--735, 2010.

\bibitem{Eckhoff}
J{\"u}rgen Eckhoff.
\newblock Helly, {R}adon, and {C}arath\'eodory type theorems.
\newblock In {\em Handbook of convex geometry, {V}ol.\ {A}, {B}}, pages
  389--448. North-Holland, Amsterdam, 1993.

\bibitem{Gortler-Thurston2}
Steven~J. Gortler and Dylan~P. Thurston.
\newblock Characterizing the {U}niversal {R}igidity of {G}eneric {F}rameworks.
\newblock {\em Discrete Comput. Geom.}, 51(4):1017--1036, 2014.

\bibitem{GLL}
Sander Gribling, David de~Laat, and Monique Laurent.
\newblock Matrices with high completely positive semidefinite rank.
\newblock {\em arXiv:1605.00988}, 2016.

\bibitem{Horn-Johnson}
Roger~A. Horn and Charles~R. Johnson.
\newblock {\em Matrix analysis}.
\newblock Cambridge University Press, Cambridge, second edition, 2013.

\bibitem{Jordan-universal-line}
Tibor Jord{\'a}n and Viet-Hang Nguyen.
\newblock On universally rigid frameworks on the line.
\newblock {\em Technical Report TR-2012-10, Egerv\'ary Research Group on
  Combinatorial Optimization, Budapest Hungary}, 2012.

\bibitem{Matousek}
Ji{\v{r}}{\'{\i}} Matou{\v{s}}ek.
\newblock {\em Lectures on discrete geometry}, volume 212 of {\em Graduate
  Texts in Mathematics}.
\newblock Springer-Verlag, New York, 2002.

\bibitem{ramana}
Motakuri~V. Ramana.
\newblock An exact duality theory for semidefinite programming and its
  complexity implications.
\newblock {\em Math. Program.}, 77(2, Ser. B):129--162, 1997.

\bibitem{Ye-So-sensor}
Anthony Man-Cho So and Yinyu Ye.
\newblock Theory of semidefinite programming for sensor network localization.
\newblock {\em Math. Program.}, 109(2-3, Ser. B):367--384, 2007.

\bibitem{ts}
Boris Tsirelson.
\newblock Quantum {B}ell-type inequalities.
\newblock {\em Hadronic Journal Supplement}, 8:329--345, 1993.

\bibitem{Whiteley-bipartite}
Walter Whiteley.
\newblock Infinitesimal motions of a bipartite framework.
\newblock {\em Pacific J. Math.}, 110(1):233--255, 1984.

\end{thebibliography}

\end{document}